\documentclass{amsart}

\usepackage{amsmath}
\usepackage{amssymb}
\usepackage{hyperref}
\usepackage{enumitem,xcolor}
\numberwithin{equation}{section} 
\usepackage{xfrac}

\newcommand{\rn}[1]{%
  \textup{\uppercase\expandafter{\romannumeral#1}}%
}

\newtheorem{theorem}{Theorem}[section]
\newtheorem{lemma}[theorem]{Lemma}
\newtheorem{corollary}[theorem]{Corollary}
\newtheorem{proposition}[theorem]{Proposition}

\theoremstyle{definition}
\newtheorem{definition}[theorem]{Definition}

\theoremstyle{remark}

\newcommand\R{\mathbb{R}}
\newcommand\Z{\mathbb{Z}}

\newcommand\C{\mathbb{C}}

\newcommand{\qtq}[1]{\quad\text{#1}\quad}

\renewcommand{\epsilon}{\varepsilon}

\newcommand{\schem}{\text{\O}}

\allowdisplaybreaks

\begin{document}

\title[Dispersive decay for the energy-critical NLS]{Dispersive decay for the energy-critical nonlinear Schr\"odinger equation}

\author[M.~Kowalski]{Matthew Kowalski}
\address{Department of Mathematics, University of California, Los Angeles, CA 90095, USA}
\email{mattkowalski@math.ucla.edu}

\begin{abstract}
    We prove pointwise-in-time dispersive decay for solutions to the energy-critical nonlinear Schr\"odinger equation in spatial dimensions $d = 3,4$ for both the initial-value and final-state problems.
\end{abstract}

\maketitle

\section{Introduction}\label{intro}
    This paper studies the asymptotic behavior of solutions to the {\em energy-critical nonlinear Schr\"odinger equation}:
\begin{equation}\label{NLS}\tag{NLS}
    \begin{cases} iu_t + \Delta u \pm |u|^\frac{4}{d-2}u = 0,\\
    u(0,x) = u_0(x) \in \dot{H}^1(\R^d),
    \end{cases}
\end{equation}
where $u(t,x)$ is a complex-valued function on spacetime $\R_t\times\R_x^d$. With this convention, $+$ represents the focusing equation and $-$ represents the defocusing equation. 

This equation is {\em energy-critical}: the scaling symmetry of \eqref{NLS},
\begin{equation*}
    u(t,x) \mapsto \lambda^{\frac{d-2}{2}} u(\lambda^2 t, \lambda x) \qtq{for} \lambda > 0,
\end{equation*}
preserves the (conserved) energy
\begin{equation*}
    E(u) = \int \tfrac{1}{2} |\nabla u(t,x)|^2 \mp \tfrac{d-2}{2d} |u(t,x)|^\frac{2d}{d-2} dx.
\end{equation*}
We discuss \eqref{NLS} for initial data in the Sobolev space $\dot{H}^1$ and the Besov space $\dot{B}^1_{2,1}$, both of which are scaling-critical.

In the seminal work \cite{defocusing-d3}, Colliander, Keel, Staffilani, Takaoka, and Tao proved global well-posedness of \eqref{NLS} in the defocusing case in spatial dimension $d= 3$. Their argument was generalized in \cite{defocusing-d4,defocusing-dgeq5} to show global well-posedness for spatial dimensions $d \geq 4$, and simplified proofs were presented in the subsequent works \cite{defocusing-simplified-d3,defocusing-simplified-d4}.


For the focusing case, global well-posedness is conjectured to hold for initial data $u_0 \in \dot{H}^1$ which satisfies $\|u_0\|_{\dot{H}^1} < \|W\|_{\dot{H}^1}$ and $E(u_0) < E(W)$, where $W$ is a stationary solution to \eqref{NLS} given by
\begin{equation*}
    W(x) = \big(1 + \tfrac{1}{d(d+2)}|x|^2\big)^{\frac{2-d}{2}}.
\end{equation*}
This conjecture was resolved for radial initial data in spatial dimension $d = 3$, see \cite{focusing-d345}, and for general initial data in spatial dimensions $d \geq 4$, see \cite{focusing-d4,focusing-dgeq5}.


In the following theorem \cite{defocusing-d3,focusing-d4,focusing-d345,defocusing-d4, defocusing-dgeq5}, we summarize the well-posedness results that are needed:
\begin{theorem}[Well-posedness]\label{well-posedness}
    Fix $d \geq 3$ and let $u_0 \in \dot{H}^1(\R^d)$. In the focusing case, assume that $u_0$ satisfies $\|u_0\|_{\dot{H}^1} < \|W\|_{\dot{H}^1}$ and $E(u_0) < E(W)$.
    In the $d = 3$ focusing case, further assume that $u_0$ is radial. Then there exists a unique global solution $u \in C_t\dot{H}_x^1$ to \eqref{NLS} with initial data $u_0$ which satisfies
    \begin{equation}\label{intro/bounds}
        \int_\R \int_{\R^d} |u(t,x)|^{\frac{2(d+2)}{d-2}} dx dt \leq C\big(\|u_0\|_{\dot{H}^1}\big).
    \end{equation}
    Moreover, there exist scattering states $u_{\pm} \in \dot{H}^1$ such that
    \begin{equation}\label{intro/scattering}
        \lim_{t \to \pm \infty}\|u(t) - e^{it\Delta} u_{\pm}\|_{\dot{H}_x^1} = 0.
    \end{equation}

\end{theorem}

The property \eqref{intro/scattering} is widely known as {\em scattering}. It indicates that solutions to the nonlinear equation \eqref{NLS} asymptotically resemble solutions to the linear Schr\"odinger equation. This then begs the question of which properties of the linear Schr\"odinger flow are exhibited by solutions to \eqref{NLS}. In particular, one characteristic of the linear Schr\"odinger equation is {\em dispersive decay}:
\begin{equation}\label{intro/linear dispersive decay}
    \|e^{it\Delta} f\|_{L_x^p} \lesssim_{p,d} |t|^{-d(\frac{1}{2} - \frac{1}{p})} \|f\|_{L^{p'}},
\end{equation}
for $2 \leq p \leq \infty$. This is derived by interpolating between the conservation of mass,
\begin{equation}\label{intro/conservation of mass}
    \|e^{it\Delta} f\|_{L^2} = \|f\|_{L^2},
\end{equation}
which follows from the Plancherel theorem, and the dispersive estimate,
\begin{equation}\label{intro/pointwise decay}
    \|e^{it\Delta} f\|_{L_x^\infty} \lesssim |t|^{-d/2} \|f\|_{L^1},
\end{equation}
which follows from the fundamental solution to the linear Schr\"odinger equation.

In light of scattering \eqref{intro/scattering} and the linear dispersive decay \eqref{intro/linear dispersive decay}, it is natural to wonder whether solutions to \eqref{NLS} also exhibit dispersive decay. We answer this for spatial dimensions $d = 3,4$ in the following theorem. 
We focus on dimensions $d = 3,4$ as these lead to polynomial nonlinearities which are the most physical. In addition, complications arise for dimensions $d \geq 5$ due to the low power of the nonlinearity; see \eqref{integrable/smallness} and the following note. 
\begin{theorem}\label{theorem}
    Fix $d \in \{3,4\}$ and $p$ such that
    \begin{equation*}\begin{cases}
        2 < p \leq \infty, & \text{ if } d = 3 \\
        2 < p < \infty, & \text{ if } d = 4.
    \end{cases}\end{equation*}
    Given $u_0 \in L^{p'} \cap \dot{H}^1(\R^d)$ satisfying the hypotheses of Theorem \ref{well-posedness}, let $u(t)$ denote the unique global solution to \eqref{NLS} with initial data $u_0$. Then
    \begin{equation}\label{theorem/decay}
        \|u(t)\|_{L_x^p} \leq C\big(\|u_0\|_{\dot{H}^1},d,p\big) |t|^{-d(\frac{1}{2} - \frac{1}{p})} \|u_0\|_{L^{p'}}
    \end{equation}
    uniformly for $t \neq 0$.
\end{theorem}
We believe that the preceding theorem is the optimal nonlinear analogue of dispersive decay \eqref{intro/linear dispersive decay}. Notably, it only requires that initial data lie in the scaling-critical space $\dot{H}^1$, it recovers a linear dependence on the initial data, and it has constants which depend only on the size of the initial data.

Absent from Theorem \ref{theorem} is $L^\infty$-decay for spatial dimension $d = 4$. The methods presented here fail for that case; see the discussion preceding Lemma \ref{d3/Linfty Sobolev} and the progression from \eqref{d4/failure 1} to \eqref{d4/failure 2}. In general, none of our methods appear able to show decay which is $|t|^{-1 - s_c}$ or faster for initial data in a critical Sobolev space $\dot{H}^{s_c}$. This restriction also appears in \cite{mc-2024}, which studies the mass-critical nonlinear Schr\"odinger equation.

In spite of this, if we restrict our initial data to the Besov space $\dot{B}^1_{2,1}$, we can show full dispersive decay for spatial dimension $d = 4$. This Besov space is scaling-critical, but is strictly stronger than the Sobolev space $\dot{H}^1$. In particular, $\dot{B}^1_{2,1} \hookrightarrow \dot{H}^1$. As an immediate corollary of this Besov result, we gain full dispersive decay for initial data in $\dot{H}^\alpha \cap \dot{H}^\beta$ for $\alpha < 1 < \beta$; see Corollary \ref{edge/corollary}. 
\begin{theorem}\label{edge case theorem}
    Given $u_0 \in L^{1} \cap \dot{B}^1_{2,1}(\R^4)$ satisfying the hypotheses of Theorem \ref{well-posedness}, let $u(t)$ denote the unique global solution to \eqref{NLS} with initial data $u_0$. Then
    \begin{equation}\label{intro/edge case decay}
        \|u(t)\|_{L_x^\infty} \leq C\big(\|u_0\|_{\dot{B}^1_{2,1}},p\big) |t|^{-2} \|u_0\|_{L^{1}}
    \end{equation}
    uniformly for $t \neq 0$.
\end{theorem}

With scattering \eqref{intro/scattering} established, it is natural to investigate the final-state problem, where a scattering state $u_{\pm}$ is given and $u$ is sought. Indeed, this is reminiscent of particle experiments where particles are made to interact and the resulting states are measured. Standard arguments then imply that this final-state problem is globally well-posed and scatters for initial data which satisfies the hypotheses of Theorem \ref{well-posedness}.

For solutions to the final-state problem, we then establish full dispersive estimates, analogous to Theorems \ref{theorem} and \ref{edge case theorem}. Remarkably, though the solution may reach a singularity in $L^p$ at the interaction time $t = 0$, dispersive decay persists through this singularity.

\begin{theorem}\label{scattering theorem}
    Fix $d \in \{3,4\}$ and $p$ such that
    \begin{equation*}\begin{cases}
        2 < p \leq \infty, & \text{ if } d = 3 \\
        2 < p < \infty, & \text{ if } d = 4.
    \end{cases}\end{equation*}
    Given $u_{\pm} \in L^{p'} \cap \dot{H}^1(\R^d)$ satisfying the hypotheses of Theorem \ref{well-posedness}, let $u(t)$ denote the unique global solution to \eqref{NLS} with final-state $u_\pm$. Then
    \begin{equation*}
        \|u(t)\|_{L^p} \leq C\big(\|u_{\pm}\|_{\dot{H}^1(\R^d)},d,p\big) |t|^{-d(\frac{1}{2} - \frac{1}{p})} \|u_\pm\|_{L^{p'}}
    \end{equation*}
    uniformly for $t \neq 0$.
\end{theorem}
\begin{theorem}\label{scattering edge case theorem}
    Given $u_\pm \in L^{1} \cap \dot{B}^1_{2,1}(\R^4)$ satisfying the hypotheses of Theorem \ref{well-posedness}, let $u(t)$ denote the unique global solution to \eqref{NLS} with final-state $u_\pm$. Then
    \begin{equation*}
        \|u(t)\|_{L^\infty} \leq C\big(\|u_\pm\|_{\dot{B}^1_{2,1}},p\big) |t|^{-\frac{d}{2}} \|u_\pm\|_{L^{1}}
    \end{equation*}
    uniformly for $t \neq 0$.
\end{theorem}

The arguments for the final-state problem are nearly identical to the arguments for the initial-value problem in Theorems \ref{theorem} and \ref{edge case theorem}. As such, we do not restate the entire proof; instead, in Section \ref{scattering} we present an example of the necessary changes.

\subsection*{Prior work}
    Prior to the development of Strichartz estimates, dispersive decay of the form \eqref{intro/linear dispersive decay} was the primary tool for understanding long-time behavior of solutions.
    As such, these estimates were closely tied to the study of well-posedness and scattering and required smooth initial data; see 
    \cite{cubic-1986, we-1985, general-1983, cubic-1978, segal} for examples.
    
    Given recent successes in scaling-critical well-posedness and scattering (such as Theorem \ref{well-posedness}), subsequent work has significantly lowered the regularity required for dispersive decay.
    However, until \cite{mc-2024}, these results have required strictly more regularity than well-posedness or scattering and have been unable to demonstrate a linear dependence on the initial data in the sense of \eqref{intro/linear dispersive decay}.

    In \cite{mc-2024}, Fan, Killip, Vi\c{s}an, and Zhao demonstrated dispersive decay for the mass-critical nonlinear Schr\"odinger equation for initial data in the scaling-critical space $L^2$.
    In the same sense as Theorem \ref{theorem}, the result of \cite{mc-2024} is the optimal nonlinear analogue of \eqref{intro/linear dispersive decay} for the mass-critical nonlinear Schr\"odinger equation.
    In Section \ref{integrable}, we adapt the methods of \cite{mc-2024} to prove Theorem \ref{theorem} for $2 < p < \frac{2d}{d-2}$.
    For $p \geq \frac{2d}{d-2}$, the methods of \cite{mc-2024} no longer apply to \eqref{NLS} and so we present a more nuanced argument which exploits the energy-criticality of our model.
    
    For our model \eqref{NLS} in particular, dispersive decay was shown for initial data with $H^{1+}$ regularity in \cite{ec-2024}.
    We note that \cite{ec-2024} worked primarily on the hyperbolic space $\mathbb{H}^3$ and then remarked on the extension to $\R^3$. 
    This improved on previous results which required initial data with $H^3$ regularity, see \cite{ec-2021,ec-2022}.
    In this paper, we lower the needed regularity to the scaling-critical spaces $\dot{H}^1$ and $\dot{B}^1_{2,1}$.

    The study of dispersive decay is extensive, with many contributing authors and variations. As an overview of the topic, we direct the reader to the following sources and references therein. For work on a variety of nonlinear Schr\"odinger equations, see \cite{cubic-2022,cubic-2023,cubic-1986,general-1983,cubic-1978}. 
    For work on various wave equations, see \cite{we-1980,we-1985,general-1983,ec-we,rel-we-1972, segal}.
    For recent work on the generalized Korteweg–de Vries and Zakharov-Kuznetsov equations, see \cite{gkdv}.
    Finally, for work on a variety of completely integrable models, see \cite{dnls-2024,cmdnls-2024,kdv-2023,bo-2019}.

    In the note \cite{ec-we-2024}, we adapt the methods presented here to the energy-critical nonlinear wave equation.

    \subsection*{Acknowledgements}
    The author was supported in part by NSF grants DMS-2154022 and DMS-2054194. 
    The author is grateful to Monica Vi\c{s}an and Rowan Killip for their discussions and guidance.
    
    This project was partially carried out during the author's visit to the Nonlinear Waves and Relativity thematic program at the Erwin Schr\"odinger Institute.

    \subsection*{Notation}

We use the standard notation $A \lesssim B$ to indicate that $A \leq C B$ for some universal constant $C > 0$ that will change from line to line. If both $A \lesssim B$ and $B \lesssim A$ then we use the notation $A \sim B$. When the implied constant fails to be universal, the relevant dependencies will be indicated within the text or included as subscripts on the symbol.

We abbreviate the maximum and minimum of two numbers $a$ and $b$ as $a \vee b = \max(a,b)$ and $a \wedge b = \min(a,b)$ respectively.

When working with embeddings, we will always use $\hookrightarrow$ to mean a continuous embedding, i.e. $X \hookrightarrow Y$ if the inclusion map $X \to Y$ satisfies $\|f\|_Y \lesssim \|f\|_X$.

Our conventions for the Fourier transform are
\begin{equation*}
\widehat{f}(\xi) = \tfrac{1}{\sqrt{2\pi}}\int e^{-i \xi x} f(x) dx \quad \text{so} \quad f(x) = \tfrac{1}{\sqrt{2\pi}} \int e^{i \xi x} \widehat{f}(\xi) d\xi.
\end{equation*}
This Fourier transform is unitary on $L^2$ and yields the standard Plancherel identities.
When a function $f(t,x)$ depends on both time and space, we let $\widehat{f}(t,\xi)$ denote the Fourier transform of $f$ in only the spatial variable.

For $s \geq 0$, we define the homogeneous Sobolev space $\dot{H}^s$ as the completion of the Schwartz functions $\mathcal{S}(\R)$ with respect to the norm
\begin{equation*}
    \|f\|^2_{\dot{H}^s} = \int |\xi|^{2s} |\widehat{f}(\xi)|^2 d\xi.
\end{equation*}

With this definition of the Fourier transform, we define the {\em Littlewood--Paley projections} as follows: Let $\varphi$ denote a smooth bump function supported on $\{|\xi| \leq 2\}$ such that $\varphi(\xi) = 1$ for $|\xi| \leq 1$. For dyadic numbers $N \in 2^\Z$, we then define $P_{\leq N}$, $P_{> N}$, and $P_N$ as
\begin{align*}
    \widehat{P_{\leq N} f} & = \varphi(\xi/N) \widehat{f}(\xi) \\
    \widehat{P_{> N} f} & = [1-\varphi(\xi/N)] \widehat{f}(\xi) \\
    \widehat{P_N f}(\xi) & = [\varphi(\xi/N) - \varphi(2\xi/N)] \widehat{f}(\xi).
\end{align*}
We will often denote $P_N f = f_N$, $P_{\leq N}f = f_{\leq N}$ and $P_{>N}f = f_{>N}$. As above, when a function $f(t,x)$ depends on both time and space, we let $f_N(t,x)$ denote the Littlewood--Paley projection of $f$ in only the spatial variable $x$.

As Fourier multipliers, the Littlewood--Paley projections commute with derivative operators and the free Schr\"odinger propagator. Moreover, they are bounded on $L^p$ for all $1 \leq p \leq \infty$ and on $\dot{H}^s$ for all $s \in \R$. For $1 < p < \infty$, we have that
\begin{equation*}
    \sum_{N \in 2^\Z} P_N f \to f \quad\text{in $L^p$}.
\end{equation*}
In addition, $P_N$, $P_{\leq N}$, and $P_{>N}$ are bounded pointwise by a constant multiple of the Hardy--Littlewood maximal function,
\begin{equation*}
    |P_N f| + |P_{\leq N} f| \lesssim Mf.
\end{equation*}

Associated to the Littlewood--Paley projections are the {\em Bernstein inequalities}, which state
\begin{equation}\begin{split}\label{intro/bernstein}
    \||\nabla|^s P_N f\|_{L^p} & \sim N^s \|P_N f\|_{L^p} \\
    \||\nabla|^s P_{\leq N} f\|_{L^p} & \lesssim N^s \|P_N f\|_{L^p} \\
    \|P_N f\|_{L^p},\|P_{\leq N} f\|_{L^p} & \lesssim N^{d(\frac{1}{q} - \frac{1}{p})} \|P_N f\|_{L^q}
\end{split}\end{equation}
for all $s \geq 0$ and $1 \leq q \leq p \leq \infty$.
Additionally, we have the Littlewood--Paley square function estimate which states
\begin{equation*}
    \big\|\|f_N(x)\|_{\ell_N^2}\big\|_{L^p} \sim \|f\|_{L^p},
\end{equation*}
for $1 < p < \infty$.

Following \cite{defocusing-d3}, we use $\schem(X)$ to denote a term that is schematically like $X$. That is, a finite linear combination of terms that look like $X$ but potentially with some terms replaced by their absolute value, complex conjugate, or a Littlewood--Paley projection. 
For examples, see \eqref{edge/duhamel w} and Lemma \ref{besov/paraproduct} where we will write
\begin{equation}\label{notation/schem}
    |v + w|^2(v+w) = \sum_{j=0}^3 \schem(v^j w^{3-j}) \qtq{and} f_{\leq N/8} \cdot g_{N_1} \cdot h = \schem(f g_{N_1} h).
\end{equation}


We use $L_t^p L_x^q(T \times X)$ to denote the mixed Lebesgue spacetime norm
\begin{equation*}
    \|f\|_{L_t^p L_x^q(T\times X)} = \big\| \|f(t,x)\|_{L^q(X,dx)} \big\|_{L^p(T,dt)} = \bigg[ \int_T \bigg(\int_X |f(t,x)|^q dx\bigg)^{p/q} dt\bigg]^{1/p}.
\end{equation*}
When $p = q$, we let $L^p_{t,x} = L_t^p L_x^p$. When $X = \R^d$, we let $L_t^pL_x^q(T) = L_t^pL_x^q(T\times\R^d)$. This is generalized to mixed Lorentz spacetime norms $L_t^{p,\theta} L_x^{q,\phi}$ in the obvious way; see Definition \ref{Lorentz/definition}.
        
\section{Lorentz theory}\label{sb}
    Throughout our analysis, it will be crucial to employ Lorentz refinements of standard inequalities and methods. Here we recall the definition and properties of Lorentz spaces which are needed. For a textbook treatment of Lorentz spaces, we direct the reader to \cite{grafakos}. 

\begin{definition}[Lorentz space]\label{Lorentz/definition}
    Fix $d \geq 1$, $1 \leq p < \infty$, and $0 < q \leq \infty$. The Lorentz space $L^{p,q}$ is the space of measurable functions $f : \R^d \to \C$ which have finite quasinorm
    \begin{equation}\label{Lorentz/quasi-norm}
        \|f\|_{L^{p,q}(\R^d)} = p^{1/q} \Big\|\lambda \big|\{x \in \R^d : |f(x)| > \lambda\}\big|^{1/p} \Big\|_{L^q((0,\infty), \frac{d\lambda}{\lambda})},
    \end{equation}
    where $|*|$ denotes the Lebesgue measure on $\R^d$.
\end{definition}
It follows that $L^{p,q}$ is a quasi-Banach space for any $1 \leq p < \infty$ and $0 < q \leq \infty$. The inclusion of $0 < q < 1$ will be necessary for the proof of Lemma \ref{d4/early time}; see \eqref{d4/early time/1} and Lemma \ref{sb/nonlinearity}.

In the case of $1 < p < \infty$ and $1 \leq q \leq \infty$, we find that
\begin{equation*}
    \|f\|_{L^{p,q}} \sim_{p,q} \sup_{\|g\|_{L^{p',q'}} = 1} \bigg|\int f(x) \overline{g(x)} dx\bigg|
\end{equation*}
where $p',q'$ are the respective H\"older conjugates. Therefore for all $1 < p < \infty$ and $1 \leq q \leq \infty$, it follows that $L^{p,q}$ is normable. In the case of $p = q$,  $L^{p,p}(\R^d)$ coincides with the standard Lebesgue space $L^p(\R^d)$. We then use the convention that $L^{\infty,\infty} = L^\infty$ and leave $L^{\infty,q}$ undefined for $q < \infty$. 

From direct calculation with \eqref{Lorentz/quasi-norm}, we find that
\begin{equation}\label{Lorentz/identity}
    \big\||x|^{-d/p}\big\|_{L^{p,\infty}(\R^d)} \sim_d 1,
\end{equation}
and hence $|x|^{-d/p} \in L^{p,\infty}(\R^d)$ for all $p \geq 1$. This is the extent to which we will use the exact form of \eqref{Lorentz/quasi-norm}.

In the same manner as the sequence spaces $\ell^q$, the Lorentz spaces $L^{p,q}$ satisfy a nesting property in the second index $q$. In particular, we have the continuous embedding $L^{p,q_1} \hookrightarrow L^{p,q_2}$, i.e.
\begin{equation*}
    \|*\|_{L^{p,q_2}} \lesssim_{p,q_1,q_2} \|*\|_{L^{p,q_1}},
\end{equation*}
for all $0 < q_1 \leq q_2 \leq \infty$. 

Lorentz spaces arise most naturally as real interpolation spaces between the usual $L^p$ spaces. This is achieved through the Hunt interpolation inequality, otherwise known as the off-diagonal Marcinkiewicz interpolation theorem; see \cite{Hunt1, Hunt2}. We recall a specific case of the theorem here:
\begin{lemma}[Hunt interpolation]
    Fix $1 \leq p_1,p_2, q_1, q_2 \leq \infty$ such that $p_1 \neq p_2$ and $q_1 \neq q_2$. Let $T$ be a sublinear operator which satisfies
    \begin{equation*}
        \|Tf\|_{L^{p_i}} \lesssim_{p_i,q_i} \|f\|_{L^{q_i}}
    \end{equation*}
    for $i \in \{1,2\}$. Then for all $\theta \in (0,1)$ and all $0 < r \leq \infty$,
    \begin{equation*}
        \|Tf\|_{L^{p_\theta, r}} \lesssim_{p_\theta,q_\theta,r} \|f\|_{L^{q_\theta,r}}
    \end{equation*}
    where $\frac{1}{p_\theta} = \frac{\theta}{p_1} + \frac{1-\theta}{p_2}$ and $\frac{1}{q_\theta} = \frac{\theta}{q_1} + \frac{1-\theta}{q_2}$.
\end{lemma}

Lorentz spaces enjoy many of the standard estimates used in the Lebesgue spaces $L^p$. In particular, H\"older's inequality carries over in the following form:
\begin{lemma}[H\"older's inequality]
    Given $1 \leq p,p_1,p_2 \leq \infty$ and $0 < q, q_1, q_2 \leq \infty$ such that $\frac{1}{p} = \frac{1}{p_1} + \frac{1}{p_2}$ and $\frac{1}{q} = \frac{1}{q_1} + \frac{1}{q_2}$,
    \begin{equation*}
        \|fg\|_{L^{p,q}} \lesssim_{d,p_i,q_i} \|f\|_{L^{p_1,q_1}}\|g\|_{L^{p_2,q_2}}.
    \end{equation*}
\end{lemma}

In addition, Lorentz spaces satisfy the Young--O'Neil convolutional inequality, see \cite{young-oneil-proof,lorentz-strichartz,young-oneil-introduction}, of which the Hardy--Littlewood--Sobolev inequality is a special case:
\begin{lemma}[Young--O'Neil convolutional inequality]
    Given $1 < p,p_1,p_2 < \infty$ and $0 < q, q_1,q_2 \leq \infty$ such that
    $\frac{1}{p} + 1 = \frac{1}{p_1} + \frac{1}{p_2}$ and $\frac{1}{q} = \frac{1}{q_1} + \frac{1}{q_2}$,
    \begin{equation*}
        \|f*g\|_{L^{p,q}} \lesssim_{d,p_i,q_i} \|f\|_{L^{p_1,q_1}} \|g\|_{L^{p_2,q_2}}.
    \end{equation*}
\end{lemma}

From Hunt interpolation and the usual Sobolev embedding theorems, we also find an analog of Sobolev embedding in Lorentz spaces,
\begin{lemma}[Sobolev embedding]
    Fix $1 < p < \infty$, $s \geq 0$, and $0 < \theta \leq \infty$ such that $\frac{1}{p} + \frac{s}{d} = \frac{1}{q}$. Then
    \begin{equation*}
        \|f\|_{L^{p,\theta}(\R^d)} \lesssim_{p,s,\theta} \big\||\nabla|^s f\big\|_{L^{q,\theta}(\R^d)}.
    \end{equation*}
\end{lemma}

Finally, we may show a basic Leibniz rule in Lorentz spaces. We recall that the Schwartz functions $\mathcal{S}(\R^d)$ are dense in $L^{p,q}$ for $q \neq \infty$, see \cite{grafakos}. By the classical Leibniz rule and extending by density, we then find the following lemma:
\begin{lemma}[Leibniz rule]
    Given $1 < p,p_i < \infty$ and $0 < q,q_i < \infty$ for $i \in \{1,2,3,4\}$ such that $\frac{1}{p} = \frac{1}{p_1} + \frac{1}{p_2} = \frac{1}{p_3} + \frac{1}{p_4}$ and $\frac{1}{q} = \frac{1}{q_1} + \frac{1}{q_2} = \frac{1}{q_3} + \frac{1}{q_4}$,
    \begin{equation*}
        \|\nabla[fg]\|_{L^{p,q}} \lesssim_{d,p,q,p_i,q_i} \|\nabla f\|_{L^{p_1,q_1}}\|g\|_{L^{p_2,q_2}} + \|f\|_{L^{p_3,q_3}}\|\nabla g\|_{L^{p_4,q_4}}.
    \end{equation*}        
\end{lemma}
    \subsection{Lorentz--Strichartz estimates}
        As observed in \cite{lorentz-strichartz}, the standard proof of Strichartz estimates for \eqref{NLS} admits a Lorentz extension using the Young--O'Neil convolutional inequality in place of Hardy--Littlewood--Sobolev. Here we allow both time and space to be placed in a Lorentz space, which presents a necessary strengthening of the inequalities used in \cite{mc-2024,lorentz-strichartz}.

\begin{definition}[Schr\"odinger-admissible]
    Fix a spatial dimension $d \geq 3$. We say that a pair $2 \leq p, q \leq \infty$ is \emph{Schr\"odinger-admissible} if
    \begin{equation*}\label{SAP}
        \frac{2}{p} + \frac{d}{q} = \frac{d}{2}.
    \end{equation*}
    We say that $(p,q)$ is a \emph{non-endpoint Schr\"odinger-admissible pair} if $2 < p,q < \infty$. Finally, we say that $(p,q)$ is \emph{Schr\"odinger-admissible with $s$ spatial derivatives} if
    \begin{equation*}
        \frac{2}{p} + d\bigg(\frac{1}{q} + \frac{s}{d}\bigg) = \frac{d}{2}.
    \end{equation*}
\end{definition}

\begin{proposition}[Lorentz--Strichartz estimates]\label{sb/Lorentz-Strichartz estimates}
    Suppose that $ 2 < p,q < \infty$ is a Schr\"odinger-admissible pair.
    Then for all $f \in L^2$ and any spacetime slab $J \times \R^d$, the linear evolution satisfies
    \begin{align}\label{sb/linear Strichartz}
        \big\|e^{it\Delta} f\big\|_{L_t^{p,2} L_x^{q,2}(J)} & \lesssim_{p,q} \|f\|_{L^2(\R^d)}.
    \end{align}
    Moreover, for all $0 < \theta \leq \infty$; $1 \leq \phi \leq \infty$; and any time-dependent interval $I(t) \subset J$,
    \begin{align}\label{sb/nonlinear Strichartz}
        \bigg\|\int_{I(t)} e^{i(t-s)\Delta}F(s,x) ds\bigg\|_{L_t^{p,\theta}L_x^{q,\phi}(J)} & \lesssim_{p,q,\theta,\phi} \|F\|_{L_t^{p',\theta}L_x^{q',\phi}(J)}.
    \end{align}
\end{proposition}

\begin{proof}
    We argue akin to the usual proof of Strichartz estimates and begin with linear dispersive decay. Applying the Hunt interpolation inequality to \eqref{intro/conservation of mass} and \eqref{intro/pointwise decay}, we find that for $2 < q < \infty$ and $0 < \phi \leq \infty$,
    \begin{equation}\label{sb/Lorentz dispersive}
        \|e^{it\Delta} f\|_{L_x^{q,\phi}} \lesssim_{q,\phi} |t|^{-d(\frac{1}{2} - \frac{1}{q})}\|f\|_{L_x^{q',\phi}}.
    \end{equation}
    
    We now prove \eqref{sb/nonlinear Strichartz}. For $1 < p < \infty$ and $1 \leq \phi \leq \infty$, we recall that $L_x^{p,\phi}$ is normable. Then for any time-dependent interval $I(t) \subset J$ and any spacetime slab $J \times \R^d$, the triangle inequality and \eqref{sb/Lorentz dispersive} imply
    \begin{align*}
        \bigg\|\int_{I(t)} \hspace{-1pt}e^{i(t-s)\Delta}F(s,x) ds \bigg\|_{L_t^{p,\theta}L_x^{q,\phi}(J)} & \lesssim_{q,\phi} \bigg\| \int_J \big\|e^{i(t-s)\Delta}F(s,x)\big\|_{L_x^{q,\phi}}ds \bigg\|_{L_t^{p,\theta}(J)}  \\
        & \lesssim_{q,\phi} \bigg\| \int_J |t-s|^{-d(\frac{1}{2} - \frac{1}{q})} \|F(s,x)\|_{L_x^{q',\phi}}ds \bigg\|_{L_t^{p,\theta}(J)}.
    \end{align*}
    For all $0 < \theta \leq \infty$, the Young-O'Neil convolutional inequality and \eqref{Lorentz/identity} then imply
    \begin{align*}
        \bigg\|\int_{I(t)} e^{i(t-s)\Delta}F(s,x) ds \bigg\|_{L_t^{p,\theta}L_x^{q,\phi}(J)} 
        & \lesssim_{p,q,\theta,\phi} \Big\||t|^{-d(\frac{1}{2} - \frac{1}{q})}\Big\|_{L_t^{\frac{2q}{d(q-2)},\infty}} \|F\|_{L_t^{p',\theta}L_x^{q',\phi}(J)} \\
        & \lesssim_d \|F\|_{L_t^{p',\theta}L_x^{q',\phi}(J)},
    \end{align*}
    which concludes the proof of \eqref{sb/nonlinear Strichartz}.
    
    To prove \eqref{sb/linear Strichartz}, we proceed with a $TT^*$ argument. Define $T : L_x^2 \to L_t^{p,2}L_x^{q,2}$ by
    \begin{equation*}
         [T f](t,x) = [e^{it\Delta} f](x).
    \end{equation*}
    Then
    $TT^* : L_t^{p',2}L_x^{q',2} \to L_t^{p,2}L_x^{q,2}$ is given by
    \begin{equation*}
        [TT^*F](t,x) = \int e^{i(t-s)\Delta} F(s,x) ds.
    \end{equation*}
    Applying \eqref{sb/nonlinear Strichartz} with $I(t) = \R$ and $\phi = \theta = 2$ to $TT^*$, we then find that $TT^*$ is bounded $L_t^{p',2}L_x^{q',2} \to L_t^{p,2}L_x^{q,2}$. This implies that $T$ is bounded $L_x^2 \to L_t^{p,2}L_x^{q,2}(J)$ and hence concludes the proof of the proposition.
\end{proof}
    \subsection{Lorentz spacetime bounds}
        We may now prove global bounds in mixed Lorentz spacetime norms for solutions to \eqref{NLS}. We present the proof for all spatial dimensions $d \geq 3$ and all non-endpoint Schr\"odinger-admissible pairs.
\begin{proposition}[Spacetime bounds]\label{spacetime bounds}
    Fix $d \geq 3$ and $\phi,\theta \geq 2$. Suppose that $2 < p,q < \infty$ is a Schr\"odinger-admissible pair and suppose that $u_0 \in \dot{H}^1(\R^d)$ satisfies the hypotheses of Theorem \ref{well-posedness}. Then the corresponding global solution $u(t)$ to \eqref{NLS} with initial data $u_0$ satisfies
    \begin{equation*}
        \|\nabla u\|_{L_t^{p,\theta}L_x^{q,\phi}} \leq C(\|u_0\|_{\dot{H}^1}).
    \end{equation*}

    The same estimate holds for the final-state problem with $u_0$ replaced by $u_{\pm}$.
\end{proposition}
\begin{proof}
    We focus on the initial-value problem first before remarking on the needed changes for the final-state problem. 
    
    It is well-known that Strichartz estimates and \eqref{intro/bounds} imply spacetime bounds for all Schr\"odinger-admissible pairs, see e.g.~\cite{defocusing-d3, defocusing-d4}. Thus for all Schr\"odinger-admissible pairs $(p,q)$,
    \begin{equation}\label{sb/intermediate bounds}
        \|\nabla u\|_{L_t^p L_x^q} \leq C\big(\|u_0\|_{\dot{H}^1}\big).
    \end{equation}

    We then turn our attention to the Lorentz case and recall the Duhamel formula:
    \begin{equation}\label{sb/duhamel}
        u(t) = e^{it\Delta}u_0 \mp i \int_{0}^t e^{i(t-s)\Delta} \big[|u|^{\frac{4}{d-2}}u\big](s) ds.
    \end{equation}
    For all $\theta,\phi \geq 2$, Proposition \ref{sb/Lorentz-Strichartz estimates} then implies
    \begin{equation*}\begin{split}
        \|\nabla u\|_{L_t^{p,\theta}L_x^{q,\phi}} 
        & \lesssim \|\nabla u_0\|_{L_x^2} + \big\|\nabla|u|^\frac{4}{d-2} u\big\|_{L_t^{p',\theta}L_x^{q',\phi}}.
    \end{split}\end{equation*}
    Consider an arbitrary non-endpoint Schr\"odinger-admissible pair $(p,q)$. As $p,q > 2$, it follows that $p' < \theta$ and $q' < \phi$. By the nesting of Lorentz spaces, we may then estimate
    \begin{equation*}\begin{split}
        \|\nabla u\|_{L_t^{p,\theta}L_x^{q,\phi}} 
        & \lesssim \|\nabla u_0\|_{L_x^2} + \big\|\nabla|u|^\frac{4}{d-2} u\big\|_{L_t^{p'}L_x^{q'}} \\
        & \lesssim \|\nabla u_0\|_{L_x^2} + \|\nabla u\|_{L_t^{p}L_x^{q}}\big\|u\big\|_{L_t^{\frac{4p}{(d-2)(p-2)}}L_x^\frac{4q}{(d-2)(q-2)}}^\frac{4}{d-2}.
    \end{split}\end{equation*}
    A quick calculation shows that $\big(\frac{4p}{(d-2)(p-2)},\frac{4q}{(d-2)(q-2)}\big)$ is a non-endpoint Schr\"odinger-admissible pair with one spatial derivative. With \eqref{sb/intermediate bounds}, this concludes the proof of the proposition for the initial-value problem.

    To adapt the preceding proof to the final-state problem, the only modifications needed are to adjust the interval of integration in \eqref{sb/duhamel} to $(-\infty, t)$ for $u_-$ and to $(t,\infty)$ with an added negative sign for $u_+$. All other changes are only in notation.
\end{proof}

An unfortunate weakness of Proposition \ref{spacetime bounds} is the inability to control Lorentz exponents below $2$. This will be necessary in the proof of Lemma \ref{d4/early time}; see \eqref{d4/early time/1}. Though this level of control appears inaccessible for the linear evolution, with \eqref{sb/nonlinear Strichartz} we can gain additional control over the nonlinear correction.
\begin{corollary}[Nonlinear correction bounds]\label{sb/nonlinearity}
    Fix $d \geq 3$, $\theta \geq \frac{2(d-2)}{d+2}$, and $\phi \geq \max\big(\frac{2(d-2)}{d+2},1\big)$. Suppose that $2 < p,q < \infty$ is a Schr\"odinger-admissible pair, and suppose that $u_0 \in \dot{H}^1(\R^d)$ satisfies the hypotheses of Theorem \ref{well-posedness}. Then the corresponding global solution $u(t)$ to \eqref{NLS} with initial data $u_0$ satisfies
    \begin{equation*}
        \bigg\|\nabla \int_{0}^t e^{i(t-s)\Delta}\big[|u|^\frac{4}{d-2}u \big](s) ds\bigg\|_{L_t^{p,\theta}L_x^{q,\phi}(\R\times\R^d)} \leq C(\|u_0\|_{\dot{H}^1}).
    \end{equation*}
    The same estimate holds for the final-state problem with the interval of integration changed to $(-\infty,t)$ {\em (}resp.\ $(t,\infty)${\em)} and $u_0$ replaced by $u_-$ {\em(}resp.\ $u_+${\em)}.
\end{corollary}
\begin{proof}
    We focus on the initial-value problem first before remarking on the needed changes for the final-state problem. 
    
    Applying the Strichartz estimate \eqref{sb/nonlinear Strichartz} and the nesting of Lorentz spaces, we find
    \begin{equation}\begin{split}\label{sb/nonlinearity 1}
        \bigg\|\nabla \int_{0}^t & e^{i(t-s)\Delta} \big[|u|^\frac{4}{d-2}u \big](s) ds\bigg\|_{L_t^{p,\theta}L_x^{q,\phi}(\R\times\R^d)} \\
        & \lesssim \Big\|\nabla |u|^{\frac{4}{d-2}}u \Big\|_{L_t^{\frac{p}{p-1},\theta}L_x^{\frac{q}{q-1},\phi}} \\
        & \lesssim \|\nabla u\|_{L_t^{p,\frac{\theta(d+2)}{d-2}}L_x^{q,\frac{\phi(d+2)}{d-2}}}\|u\|_{L_t^{\frac{4p}{(d-2)(p-2)},\frac{\theta(d+2)}{d-2}}L_x^{\frac{4q}{(d-2)(q-2)},\frac{\phi(d+2)}{d-2}}}^\frac{4}{d-2} \\
        & \lesssim \|\nabla u\|_{L_t^{p,2}L_x^{q,2}}\|u\|_{L_t^{\frac{4p}{(d-2)(p-2)},2}L_x^{\frac{4q}{(d-2)(q-2)},2}}^\frac{4}{d-2}.
    \end{split}\end{equation}
    A quick calculation shows that $\big(\frac{4p}{(d-2)(p-2)},\frac{4q}{(d-2)(q-2)}\big)$ is a non-endpoint Schr\"odinger-admissible pair with one spatial derivative. Proposition \ref{spacetime bounds} then concludes the proof of the proposition for the initial-value problem.

    To adapt the preceding proof to the final-state problem, the only modification needed is to change the interval of integration in \eqref{sb/nonlinearity 1} from $(0,t)$ to $(-\infty, t)$ for $u_-$ and from $(0,t)$ to $(t,\infty)$ for $u_+$.
\end{proof}

\section{Proof of Theorem \ref{theorem}}
    We decompose the proof of Theorem \ref{theorem} into the cases $2 < p < \frac{2d}{d-2}$ and $\frac{2d}{d-2} \leq p$. For $2 < p < \frac{2d}{d-2}$, we find that the linear dispersive decay \eqref{intro/linear dispersive decay} is integrable near $t = 0$. We therefore call $2 < p < \frac{2d}{d-2}$ the {\em integrable case} of Theorem \ref{theorem}. This integrability leads to a simplified argument which parallels the proof in \cite{mc-2024} for the mass-critical nonlinear Schr\"odinger equation. We present this proof in Section \ref{integrable}.
    
    For $p \geq \frac{2d}{d-2}$, the linear dispersive decay \eqref{intro/linear dispersive decay} is no longer integrable near $t = 0$ and so a more nuanced argument is needed. This will be completed in Section \ref{bp}.
    \subsection{Integrable case} \label{integrable}
        In this section, we prove Theorem \ref{theorem} in the case where $2 < p < \frac{2d}{d-2}$. 
This proof forms a template for the remaining cases of Theorems \ref{theorem}, \ref{edge case theorem}, \ref{scattering theorem}, and \ref{scattering edge case theorem}. In addition, the result in this section will be used in the proofs of Lemma \ref{d3/early time} and Lemma \ref{d4/early time}. 

This section focuses only on the initial-value problem \eqref{NLS}. The final-state problem is discussed in Section \ref{scattering} where we present the proof of Theorem \ref{scattering theorem} for $d = 3$, $p = 3$ specifically. As such, for an example of (essentially) the following proof with explicit numbers, we direct the reader to Section \ref{scattering}.
\begin{proof}[Proof of the integrable case of Theorem \ref{theorem}]
    It suffices to work with $t > 0$ as $t < 0$ will follow from time-reversal symmetry. By the density of Schwartz functions in $\dot{H}^1 \cap L^{p'}$, it suffices to consider Schwartz solutions of \eqref{NLS}.
    
    For $0 < T \leq \infty$, we define the norm
    \begin{equation*}
        \|u\|_{X(T)} = \sup_{t \in [0,T)} |t|^{d(\frac{1}{2} - \frac{1}{p})}\|u(t)\|_{L_x^p}.
    \end{equation*}
    It then suffices to show
    \begin{equation}\label{integrable/conclusion}
        \|u\|_{X(\infty)} \leq C(\|u_0\|_{\dot{H}^1})\|u_0\|_{L^{p'}},
    \end{equation}
    for which we proceed with a bootstrap argument.

    Let $\eta > 0$ denote a small parameter to be chosen later, depending only on universal constants. A quick calculation shows that 
    \begin{equation*}
        (q,r) = \big(\tfrac{8p}{(d-2)(2d - (d-2)p)},\tfrac{4p}{(p-2)(d-2)}\big)
    \end{equation*}
    is a non-endpoint Schr\"odinger-admissible pair with one spatial derivative. Proposition \ref{spacetime bounds} then implies that we may decompose $[0,\infty)$ into $J = J(\|u_0\|_{\dot{H}^1},\eta)$ many intervals $I_j = [T_{j-1},T_j)$ on which
    \begin{equation}\label{integrable/smallness}
        \|u\|_{L_s^{q,\frac{4}{d-2}} L_x^r(I_j)}
        < \eta.
    \end{equation}
    We note that this relies on $q = \frac{8p}{(d-2)(2d - (d-2)p)} < \infty$, which requires $p < \frac{2d}{d-2}$, and on $\frac{4}{d-2} \geq 2$, which requires $d \in \{3,4\}$.
    
    We aim to show that for all $j = 1,\dots, J$,
    \begin{equation}\label{integrable/bootstrap}
        \|u\|_{X(T_j)} \lesssim \|u_0\|_{L^{p'}} + C(\|u_0\|_{\dot{H^1}})\|u\|_{X(T_{j-1})} + \eta^\frac{4}{d-2} \|u\|_{X(T_j)}.
    \end{equation}
    Choosing $\eta > 0$ sufficiently small based on the constants in \eqref{integrable/bootstrap}, we could then iterate over $j = 1,\dots, J(\|u_0\|_{\dot{H}^1})$ to yield \eqref{integrable/conclusion} and conclude the proof of the integrable case of Theorem \ref{theorem}.

    We therefore focus on \eqref{integrable/bootstrap}. Fix $t \in [0,T_j)$ and recall the Duhamel formula:
    \begin{equation*}
        u(t) = e^{it\Delta} u_0 \mp i \int_0^{t} e^{i(t-s)\Delta}\big[|u|^\frac{4}{d-2}u\big](s)ds.
    \end{equation*}
    By the linear dispersive decay \eqref{intro/linear dispersive decay}, the contribution of the linear term to $\|u(t)\|_{X(T_j)}$ is immediately seen to be acceptable:
    \begin{equation}\label{integrable/linear}
        \big\|e^{it\Delta} u_0\big\|_{X(T_j)} \lesssim \|u_0\|_{L_x^{p'}}.
    \end{equation}
    We thus focus on the nonlinear correction.
    
    By the linear dispersive decay \eqref{intro/linear dispersive decay} and H\"older's inequality, we may estimate
    \begin{align*}
         \bigg\|\int_0^t e^{i(t-s)\Delta}\big[|u|^\frac{4}{d-2}u\big](s)ds\bigg\|_{L_x^p}
         \lesssim \int_0^t |t-s|^{-d(\frac{1}{2} - \frac{1}{p})}\|u(s)\|_{L^p} \big\|u(s)\big\|^\frac{4}{d-2}_{L_x^r}ds.
    \end{align*}
    By definition, $|s|^{d(\frac{1}{2}-\frac{1}{p})}\|u(s)\|_{L^p} \leq \|u\|_{X(s)}$. Then
    \begin{equation}\begin{split}\label{integrable/bootstrap introduction}
        \bigg\| \int_0^t e^{i(t-s)\Delta} & \big[|u|^\frac{4}{d-2}u\big](s)ds\bigg\|_{L_x^p} \\
        & \lesssim \int_0^t |t-s|^{-d(\frac{1}{2} - \frac{1}{p})}|s|^{-d(\frac{1}{2} - \frac{1}{p})}\|u\|_{X(s)}\|u(s)\|^\frac{4}{d-2}_{L_x^r}ds.
    \end{split}\end{equation}
    
    We decompose $ [0,t) $ into $[0,t/2)$ and $[t/2,t)$. For $s \in  [0,t/2) $, we note that $|t-s| \sim |t|$, and for $s \in [t/2,t)$, we note that $|s| \sim |t|$. Doing so, we find
    \begin{equation}\begin{split}\label{integrable/early-late}
        \bigg\|\int_0^t e^{i(t-s)\Delta} \big[|u|^\frac{4}{d-2}& u\big](s) ds\bigg\|_{L_x^p} \\
        & \lesssim |t|^{-d(\frac{1}{2} - \frac{1}{p})}\int_0^{t/2}|s|^{-d(\frac{1}{2} - \frac{1}{p})}\|u\|_{X(s)}\|u(s)\|^\frac{4}{d-2}_{L_x^r}ds\\
        &\hspace{11pt} + |t|^{-d(\frac{1}{2} - \frac{1}{p})}\int_{t/2}^t |t-s|^{-d(\frac{1}{2} - \frac{1}{p})}\|u\|_{X(s)}\|u(s)\|^\frac{4}{d-2}_{L_x^r}ds.
    \end{split}\end{equation}
    
    Because $2 < p < \frac{2d}{d-2}$, it follows that $|s|^{-d(\frac{1}{2} - \frac{1}{p})}$ and $|t-s|^{-d(\frac{1}{2} - \frac{1}{p})}$ both lie in $L_s^{\frac{2p}{d(p - 2)},\infty}$, see \eqref{Lorentz/identity}. H\"older's inequality then implies
    \begin{align*}
        \bigg\|& \int_0^t e^{i(t-s)\Delta} \big[|u|^\frac{4}{d-2}u\big](s)ds\bigg\|_{L_x^p} \lesssim |t|^{-d(\frac{1}{2} - \frac{1}{p})}\Big\|\|u\|_{X(s)}\|u(s)\|^\frac{4}{d-2}_{L_x^r}\Big\|_{L_s^{\frac{2p}{(2-d)p + 2d},1}\left( [0,t) \right)}.
    \end{align*}
    The preceding calculation relies on both that \eqref{intro/linear dispersive decay} is integrable near $t = 0$ and on the Lorentz improvements in Proposition \ref{sb/Lorentz-Strichartz estimates}.

    For $t \in [0,T_j)$, we now decompose $[0,t)$ into $[0,t) \cap[0,T_{j-1})$ and $[0,t) \cap I_j$. Doing so, \eqref{integrable/smallness} and Proposition \ref{spacetime bounds} then imply
    \begin{align*}
        \Big\|\|u\|_{X(s)}&\|u(s)\|^\frac{4}{d-2}_{L_x^r}\Big\|_{L_s^{\frac{2p}{(2-d)p + 2d},1}\left( [0,t) \right)}\\
        & \leq \|u\|_{X(T_{j-1})}\|u\|^\frac{4}{d-2}_{L_s^{q,\frac{4}{d-2}}L_x^r([0,T_{j-1}))}+ \|u\|_{X(T_j)}\|u\|^\frac{4}{d-2}_{L_s^{q,\frac{4}{d-2}}L_x^{r}(I_j)} \\
        & \leq C(\|u_0\|_{\dot{H}^1})\|u\|_{X(T_{j-1})} + \eta^\frac{4}{d-2}\|u\|_{X(T_j)}.
    \end{align*}
    Combining the two preceding estimates and taking the supremum over $t \in [0,T_j)$, we find that
    \begin{align*}
        \bigg\|\int_0^t e^{i(t-s)\Delta}\big[|u|^\frac{4}{d-2}u\big](s)ds\bigg\|_{X(T_j)}
        \lesssim C(\|u_0\|_{\dot{H}^1})\|u\|_{X(T_{j-1})} + \eta^\frac{4}{d-2}\|u\|_{X(T_j)}.
    \end{align*}
    Along with the linear term \eqref{integrable/linear}, this yields the bootstrap statement \eqref{integrable/bootstrap} and concludes the proof of the integrable case of Theorem \ref{theorem}.
\end{proof}

    \subsection{Non-integrable case}\label{bp}
        We now complete the proof of Theorem \ref{theorem} by considering $\frac{2d}{d-2} \leq p$. We follow the structure of the proof of the integrable case, with care now taken to avoid the non-integrability of the linear dispersive decay \eqref{intro/linear dispersive decay} near $t = 0$.

As in the integrable case, for each $t > 0$ we will decompose the integral over $[0,t)$ into an early-time interval $[0,t/2)$ and a late-time interval $[t/2,t)$. Unlike the integrable case, these intervals must now be treated separately.

We begin with the early-time interval, $[0,t/2)$. On this interval, we carefully apply the integrable case of Theorem \ref{theorem} to produce a factor of $\|u_0\|_{L^{p'}}$. In doing so, we avoid the bootstrap norm which would produce a non-integrable term $|s|^{-d(\frac{1}{2} - \frac{1}{p})}$; see \eqref{integrable/bootstrap introduction}. As this argument is independent of the bootstrap structure, we present this estimate in Lemma \ref{d3/early time} for $d = 3$ and in Lemma \ref{d4/early time} for $d = 4$.

We focus first on $d = 3$ as it is the simpler argument.
\begin{lemma}[Early-time interval, $d = 3$]\label{d3/early time}
    Fix $6 \leq p \leq \infty$. Suppose that $u_0 \in \dot{H}^1 \cap L^{p'}(\R^3)$ satisfies the hypotheses of Theorem \ref{well-posedness}. Then the corresponding solution $u(t)$ to \eqref{NLS} with initial data $u_0$ satisfies
    \begin{equation*}
        \bigg\|\int_0^{t/2} e^{i(t-s)\Delta}\big[|u|^4u\big](s)ds \bigg\|_{L^p} \leq C(\|u_0\|_{\dot{H}^1}) |t|^{-3(\frac{1}{2} - \frac{1}{p})}\|u_0\|_{L^{p'}}.
    \end{equation*}
\end{lemma}
\begin{proof}
    As before, we note that $|t-s| \sim |t|$ for $s \in [0,t/2)$. Then by the linear dispersive decay \eqref{intro/linear dispersive decay} and H\"older's inequality,
    \begin{align*}
        \bigg\|\int_0^{t/2} e^{i(t-s)\Delta}\big[|u|^4u\big](s)ds \bigg\|_{L^p}
        & \lesssim \int_0^{t/2} |t-s|^{-3(\frac{1}{2} - \frac{1}{p})}\big\|\big[|u|^4 u\big](s)\big\|_{L_x^{\frac{p}{p-1}}} ds \\
        & \lesssim |t|^{-3(\frac{1}{2} - \frac{1}{p})}\int_0^{t/2} \|u(s)\|_{L_x^3}^{\frac{5p-6}{3p}}\big\|u(s)\big\|^{\frac{10p + 6}{3p}}_{L_x^\frac{6(5p+3)}{4p-3}}ds
    \end{align*}
    
    By the integrable case of Theorem \ref{theorem}, H\"older's inequality and Sobolev embedding imply that
    \begin{align*}
        \|u(s)\|_{L_x^3} 
        & \leq C(\|u_0\|_{\dot{H}^1}) |s|^{-1/2} \|u_0\|_{L^{p'}}^{\frac{3p}{5p-6}}\|u_0\|_{L^6}^{\frac{2p-6}{5p-6}}\\
        & \leq C(\|u_0\|_{\dot{H}^1}) |s|^{-1/2} \|u_0\|_{L^{p'}}^{\frac{3p}{5p-6}}.
    \end{align*}
    Therefore,
    \begin{align*}
        \bigg\|\int_0^{t/2}& e^{i(t-s)\Delta}\big[|u|^4u\big](s)ds \bigg\|_{L^p} \\
        & \leq C(\|u_0\|_{\dot{H}^1})|t|^{-3(\frac{1}{2} - \frac{1}{p})}\|u_0\|_{L^{p'}}\int_0^{t/2} |s|^{-\frac{5p-6}{6p}}\big\|u(s)\big\|^{\frac{10p + 6}{3p}}_{L_x^\frac{6(5p+3)}{4p-3}}ds \\
        & \leq C(\|u_0\|_{\dot{H}^1})|t|^{-3(\frac{1}{2} - \frac{1}{p})}\|u_0\|_{L^{p'}}\big\|u\big\|^{\frac{10p + 6}{3p}}_{L_s^{\frac{4(5p+3)}{p+6},\frac{10p + 6}{3p}}L_x^\frac{6(5p+3)}{4p-3}}.
    \end{align*}
    Because $\big(\frac{4(5p+3)}{p+6},\frac{6(5p+3)}{4p-3}\big)$ is a non-endpoint Schr\"odinger-admissible pair with one spatial derivative and $\frac{10p + 6}{3p} \geq 2$, Proposition \ref{spacetime bounds} then concludes the proof of the lemma.
\end{proof}


We now turn our attention to the early-time interval $[0,t/2)$ in spatial dimension $d = 4$. Here, an additional challenge arises from the limited copies of $u$ present in the nonlinearity. This prevents the use of Proposition \ref{spacetime bounds} due to the requirement that the Lorentz exponent be above $2$; see \eqref{d4/early time/1} and the follow note.

To handle this, we decompose $u$ into its linear term and nonlinear correction:
\begin{equation}\label{d4/u decomposition}
    u(t) = e^{it\Delta} u_0 \mp i \int_0^t e^{i(t-s)\Delta}\big[|u|^2u \big](s) ds = e^{it\Delta}u_0 + u_{nl}.
\end{equation}
In light of Corollary \ref{sb/nonlinearity}, we have spacetime bounds for $u_{nl}$ down to Lorentz exponent $2/3$. This will allow us to adapt the proof of Lemma \ref{d3/early time} for the nonlinear correction $u_{nl}$.

We start with $e^{it\Delta}u_0$, which we address in the following lemma.
\begin{lemma}\label{d4/linear control}
    Fix $1 \leq q \leq 2$ and suppose that $f \in L^q \cap \dot{H}^1(\R^4)$. Then
    \begin{equation*}
        \|e^{it\Delta} f\|_{L_t^{3}L_x^{3q}} \lesssim \|f\|_{L^q}^{1/3} \|f\|_{\dot{H}^1}^{2/3}.
    \end{equation*}
\end{lemma}
\begin{proof}
    Consider a Littlewood-Paley piece $f_N$ for some $N \in 2^\Z$. Because $(3,3)$ is a Schr\"odinger-admissible pair, Strichartz estimates and the Berntein inequalities \eqref{intro/bernstein} imply that
    \begin{align*}
        \|e^{it\Delta} f_N\|_{L_t^3 L_x^{3q}} & \lesssim N^{\frac{4q-4}{3q}}\|e^{it\Delta}f_N\|_{L_{t,x}^3}\lesssim N^{\frac{4q-4}{3q}}\|f_N\|_{L^2}.
    \end{align*}
    Using the Bernstein inequalities \eqref{intro/bernstein} again, we may estimate $e^{it\Delta} f_N$ in two ways:
    \begin{align}
        \label{d4/linear control/1}
        \|e^{it\Delta} f_N\|_{L_t^3 L_x^{3q}} & \lesssim N^{\frac{q-4}{3q}}\|\nabla f_N\|_{L^2}, \\ 
        \label{d4/linear control/2}
        \|e^{it\Delta} f_N\|_{L_t^3 L_x^{3q}} & \lesssim N^{\frac{2(4-q)}{3q}}\|f_N\|_{L^q}.
    \end{align}

    We now decompose $f$ into high and low frequencies based on a cutoff $M \in 2^\Z$ to be chosen later. We apply \eqref{d4/linear control/1} to the high frequencies and \eqref{d4/linear control/2} to the low frequencies. Countable subadditivity then implies that
    \begin{align*}
        \|e^{it\Delta} f\|_{L_t^3L_x^{3q}} 
        & \lesssim \sum_{N > M}N^{\frac{q-4}{3q}}\|f_N\|_{\dot{H}^1} + \sum_{N \leq M} N^{\frac{2(4-q)}{3q}}\|f_N\|_{L^q}\\
        & \lesssim  M^{\frac{-(4-q)}{3q}}\|f\|_{\dot{H}^1} + M^{\frac{2(4-q)}{3q}} \|f\|_{L^q}.
    \end{align*}
    Choosing $M^{\frac{4-q}{q}} \sim \|f\|_{\dot{H}^1}/\|f\|_{L^q}$ then concludes the proof of the lemma.
\end{proof}

We now address the early-time interval $[0,t/2)$ in spatial dimension $d = 4$ with the following lemma. Note that the following lemma is also applicable to $p = \infty$ in $d = 4$, which will be used to prove Theorem \ref{edge case theorem}; see \eqref{edge/II}.
\begin{lemma}[Early-time interval, $d=4$]\label{d4/early time}
    Fix $4 \leq p \leq \infty$. Suppose that $u_0 \in \dot{H}^1 \cap L^{p'}(\R^4)$ satisfies the hypotheses of Theorem \ref{well-posedness} and let $u(t)$ be the solution to \eqref{NLS} with initial data $u_0$. Then
    \begin{equation*}
        \bigg\|\int_0^{t/2} e^{i(t-s)\Delta}\big[|u|^2u\big](s)ds \bigg\|_{L^p} \leq C(\|u_0\|_{\dot{H}^1}) |t|^{-4(\frac{1}{2} - \frac{1}{p})} \|u_0\|_{L^{p'}}.
    \end{equation*}
\end{lemma}
\begin{proof}
    For $s \in [0,t/2)$, we note that $|t-s| \sim |t|$. The linear dispersive decay \eqref{intro/linear dispersive decay} then implies that
    \begin{align*}
        \bigg\|\int_0^{t/2} e^{i(t-s)\Delta}\big[|u|^2u\big](s)ds \bigg\|_{L^p} & \lesssim |t|^{-4(\frac{1}{2} - \frac{1}{p})} \int_0^{t/2} \big\|\big[|u|^2u\big](s)\big\|_{L^{p'}}ds.
    \end{align*}
    We now decompose $u = e^{it\Delta}u_0 + u_{nl}$ as in \eqref{d4/u decomposition}. Expanding the cubic nonlinearity and applying Lemma \ref{d4/linear control}, we may estimate
    \begin{equation}\begin{split}\label{d4/early time/0}
        &\bigg\|\int_0^{t/2}e^{i(t-s)\Delta}\big[|u|^2u\big](s)ds \bigg\|_{L^p} \\
        & \lesssim |t|^{-4(\frac{1}{2} - \frac{1}{p})} \bigg[ \big\|e^{is\Delta} u_0 \big\|^3_{L_s^3L_x^{3p'}} + \sum_{\alpha  = 0}^2 \int_0^{t/2} \Big\|(e^{is\Delta}u_0)^\alpha u_{nl}^{3-\alpha}\Big\|_{L^\frac{p}{p-1}}ds\bigg] \\
        & \leq C(\|u_0\|_{\dot{H}^1})|t|^{-4(\frac{1}{2} - \frac{1}{p})} \bigg[\|u_0\|_{L^{p'}} + \sum_{\alpha  = 0}^2 \int_0^{t/2} \Big\|(e^{is\Delta}u_0)^\alpha u_{nl}^{3-\alpha}\Big\|_{L^\frac{p}{p-1}}ds\bigg] \\
        & = C(\|u_0\|_{\dot{H}^1})|t|^{-4(\frac{1}{2} - \frac{1}{p})} \bigg[\|u_0\|_{L^{p'}} + \sum_{\alpha  = 0}^2 I_\alpha \bigg].
    \end{split}\end{equation}
    
    We consider the integral term $I_\alpha$ and proceed as in the proof of Lemma \ref{d3/early time}. For ease of notation, we define $\beta = \frac{15p-20}{7p}$. With this definition, H\"older's inequality and Sobolev embedding ensures that for $p \geq 4$,
    \begin{equation}\label{d4/early time/holder}
        \|u_0\|_{L_x^{5/3}}^\beta \lesssim \|u_0\|_{L_x^{p'}}\|u_0\|_{L_x^4}^{\frac{(8p-20)\beta}{15p-20}} \leq C(\|u_0\|_{\dot{H}^1}) \|u_0\|_{L_x^{p'}}.
    \end{equation}

    By H\"older's inequality, we estimate $I_\alpha$ as
    \begin{align*}
        I_\alpha & \lesssim \bigg\|\big\|e^{is\Delta} u_0\big\|_{L_x^{5/2}}^{\alpha \vee \beta}\big\|u_{nl}(s)\big\|^{\beta - \alpha \vee \beta}_{L_x^{5/2}}\Big\|(e^{is\Delta} u_0)^{\alpha - \alpha \vee \beta} u_{nl}^{3 - \alpha \wedge \beta}(s)\Big\|_{L_x^\frac{7p}{p+1}}\bigg\|_{L_s^1}.
    \end{align*}
    Applying the triangle inequality, the linear dispersive decay \eqref{intro/linear dispersive decay} and the integrable case of Theorem \ref{theorem} to $u_{nl}$, we find that
    \begin{equation}\label{d4/early time/unl}
        \|u_{nl}(s)\|_{L_x^{5/2}} \leq \|e^{is\Delta} u_0\|_{L_x^{5/2}} + \|u(s)\|_{L_x^{5/2}} \leq C\big(\|u_0\|_{\dot{H}^1}\big)|s|^{-2/5} \|u_0\|_{L^{5/3}}.
    \end{equation}
    Combined with the linear dispersive decay \eqref{intro/linear dispersive decay} and \eqref{d4/early time/holder}, this then implies that
    \begin{align*}
        I_\alpha & \lesssim \bigg\| |s|^{-2\beta/5}\|u_0\|^{\beta}_{L_x^{5/3}}\big\|(e^{is\Delta} u_0)^{\alpha - \alpha \vee \beta} u_{nl}^{3 - \alpha \wedge \beta}(s)\big\|_{L_x^\frac{7p}{p+1}}\bigg\|_{L_s^1} \\
        & \lesssim C(\|u_0\|_{\dot{H}^1}) \|u_0\|_{L_x^{p'}}\bigg\| |s|^{-\frac{6p-8}{7p}}\big\|(e^{is\Delta} u_0)^{\alpha - \alpha \vee \beta} u_{nl}^{3 - \alpha \wedge \beta}(s)\big\|_{L_x^\frac{7p}{p+1}}\bigg\|_{L_s^1}.
    \end{align*}
    As $0 < \frac{6p-8}{7p} < 1$ for all $p \geq 4$, we note that $|s|^{-\frac{6p-8}{7p}} \in L^{\frac{7p}{6p-8},\infty}$; see \eqref{Lorentz/identity}. H\"older's inequality then implies that
    \begin{align*}
        I_\alpha & \lesssim C(\|u_0\|_{\dot{H}^1}) \|u_0\|_{L_x^{p'}}\big\|(e^{is\Delta} u_0)^{\alpha - \alpha \vee \beta} u_{nl}^{3 - \alpha \wedge \beta}(s)\big\|_{L_s^{\frac{7p}{p+8},1}L_x^\frac{7p}{p+1}}.
    \end{align*}
    
    We note that $\alpha - \alpha \vee \beta + 3 - \alpha \wedge \beta = 3 - \beta$. By H\"older's inequality and direct calculation, we may then estimate
    \begin{align*}
        I_\alpha & \lesssim C(\|u_0\|_{\dot{H}^1}) \|u_0\|_{L_x^{p'}}\big\|e^{is\Delta} u_0 \big\|_{L_s^{\frac{7p(3-\beta)}{p+8},\infty}L_x^\frac{7p(3-\beta)}{p+1}}^{\alpha - \alpha \vee \beta} \big\|u_{nl}\big\|_{L_s^{\frac{7p(3-\beta)}{p+8},3 - \alpha \wedge \beta}L_x^\frac{7p(3-\beta)}{p+1}}^{3 - \alpha\wedge \beta}\\
        & = C(\|u_0\|_{\dot{H}^1}) \|u_0\|_{L_x^{p'}}\big\|e^{is\Delta} u_0 \big\|_{L_s^{\frac{6p+20}{p+8},\infty}L_x^\frac{6p+20}{p+1}}^{\alpha - \alpha \vee \beta} \big\|u_{nl}\big\|_{L_s^{\frac{6p+20}{p+8},3 - \alpha \wedge \beta}L_x^\frac{6p+20}{p+1}}^{3 - \alpha\wedge \beta}.
    \end{align*}
    A quick calculation verifies that $(\frac{6p+20}{p+8},\frac{6p+20}{p+1})$ is Schr\"odinger-admissible with one spatial derivative and that
    \begin{equation}\label{d4/early time/1}
        3-\alpha \wedge \beta \geq 3 - \beta = \tfrac{6p+20}{7p} \geq \tfrac{2}{3}
    \end{equation}
    for all $p \geq 4$. We note here that the case $\beta \geq \alpha$ requires the additional Lorentz control from Corollary \ref{sb/nonlinearity}.
    
    Corollary \ref{sb/nonlinearity} and Proposition \ref{sb/linear Strichartz} then imply that
    \begin{equation*}
        I_\alpha \leq C(\|u_0\|_{\dot{H}^1})\|u_0\|_{L^{p'}},
    \end{equation*}
    for all $\alpha = 0,1,2$. Along with \eqref{d4/early time/0}, this concludes the proof of the lemma.
\end{proof}

        We now turn our attention to the late-time interval $[t/2,t)$. On this interval, we employ a Sobolev embedding before applying the linear dispersive decay \eqref{intro/linear dispersive decay}; see \eqref{bp/Sobolev}. This decreases the Lebesgue exponent below the integrable threshold $\frac{2d}{d-2}$ in all cases except $p = \infty, d= 4$. Applying the linear dispersive decay \eqref{intro/linear dispersive decay} at this point yields an integrable term which we control as in the proof of the integrable case of Theorem \ref{theorem}.

In the excluded case $p = \infty$, $d = 4$, a Sobolev embedding with one derivative is insufficient to decrease the Lebesgue exponent below the integrable threshold. This can be fixed with a careful frequency decomposition, such in the proof of Theorem \ref{edge case theorem}, or with additional regularity, such as in Corollary \ref{edge/corollary}.

In order to consider the case $p = \infty$ when $d = 3$, we must address the failure of endpoint Sobolev embedding. To circumvent this, we establish the following lemma, which acts as a combination of Sobolev embedding and linear dispersive decay \eqref{intro/linear dispersive decay}.
\begin{lemma}\label{d3/Linfty Sobolev}
    For $f \in L^1 \cap \dot{H}^1(\R^3)$,
    \begin{equation*}
        \|e^{it\Delta} f\|_{L_x^\infty} \lesssim |t|^{-1/2} \||\nabla| f\|_{L_x^{3/2,1}}.
    \end{equation*}
\end{lemma}
\begin{proof}
    For $\epsilon > 0$, we take a Gaussian approximation to the identity,
    \begin{equation*}
        \widehat{\psi}_\epsilon(\xi) = e^{-\epsilon |\xi|^2}.
    \end{equation*}
    The linear dispersive decay \eqref{intro/linear dispersive decay} then implies that
    \begin{equation}\begin{split}\label{d3/Linfty Sobolev/1}
        \|e^{it\Delta} f\|_{L_x^\infty} \lesssim \|e^{it\Delta} [\psi_\epsilon*f]\|_{L_x^\infty} + |t|^{-3/2}\|\psi_\epsilon*f - f\|_{L_x^1}. 
    \end{split}\end{equation}
    Because $f \in L^1$, the second term will vanish as $\epsilon \to 0$. It thus suffices to consider the first term.
    
    For the first term, we set up an oscillatory integral. Applying the Young--O'Neil convolutional inequality, we find that
    \begin{equation}\begin{split}\label{d3/Linfty Sobolev/2}
        \|e^{it\Delta} [\psi_\epsilon*f]\|_{L_x^\infty} & \sim \|e^{it\Delta}|\nabla|^{-1}\psi_\epsilon*[|\nabla| f]\|_{L_x^\infty} \\
        & \lesssim \bigg\|\int_{\R^3} e^{-it|\xi|^2 + ix\cdot\xi - \epsilon |\xi|^2}|\xi|^{-1} d\xi\bigg\|_{L_x^{3,\infty}}\||\nabla| f\|_{L_x^{3/2,1}}.
    \end{split}\end{equation}
    Consider the integral term. Converting to spherical coordinates with $x$ taken as the zenith direction, it follows that
    \begin{align*}
        \int_{\R^3} e^{-it|\xi|^2 + ix\cdot\xi-\epsilon|\xi|^2} |\xi|^{-1} d\xi  & = \int_0^{2\pi}\int_{-1}^1\int_0^\infty e^{-it\rho^2 + i|x|\rho \cos \phi - \epsilon \rho^2}\rho \;d\rho \;d\cos\phi \;d \theta \\
        & = \frac{2\pi}{i|x|}\lim_{R \to \infty}\int_0^R\Big( e^{-it\left[\rho^2 + \frac{|x|}{t}\rho\right]} - e^{-it\left[\rho^2 -\frac{|x|}{t}\rho\right]}\Big)e^{-\epsilon \rho^2}d\rho.
    \end{align*}
    Here the limit representation is justified because the integrand is in $L^1$. 
    
    Consider the phase factors $\rho^2 \pm \frac{|x|}{t}\rho$. For all $\rho,t,$ and $|x|$, we find
    \begin{equation*}
        \bigg|\frac{\partial^2}{\partial \rho^2}\Big(\rho^2 \pm \frac{|x|}{t}\rho\Big)\bigg| \geq 1.
    \end{equation*}
    Applying Van der Corput's lemma, see \cite[Cor.\ 2.6.8]{grafakos}, we may then estimate
    \begin{align*}
        \bigg|\int_0^R\Big( e^{-it\left[\rho^2 + \frac{|x|}{t}\rho\right]} - e^{-it\left[\rho^2 -\frac{|x|}{t}\rho\right]}\Big)e^{-\epsilon \rho^2}d\rho\bigg| & \lesssim |t|^{-1/2}\Big(e^{-\epsilon R^2} + \big\|\partial_\rho e^{-\epsilon \rho^2}\big\|_{L_\rho^1} \Big) \\
         & \lesssim |t|^{-1/2},
    \end{align*}
    uniformly for $ R > 0$ and for $\epsilon > 0$ sufficiently small.
    
    Taking the limit as $R\to \infty$, we then find that
    \begin{equation}\begin{split}\label{d3/Linfty Sobolev/3}
        \bigg\|\int_{\R^3} e^{-it|\xi|^2 + ix\xi - \epsilon |\xi|^2}|\xi|^{-1} d\xi\bigg\|_{L_x^{3,\infty}} \lesssim \big\||x|^{-1}|t|^{-1/2}\big\|_{L_x^{3,\infty}(\R^3)} \lesssim |t|^{-1/2},
    \end{split}\end{equation}
    uniformly for $\epsilon > 0$ sufficiently small. Combining the estimates \eqref{d3/Linfty Sobolev/1}, \eqref{d3/Linfty Sobolev/2}, and \eqref{d3/Linfty Sobolev/3}, we may then take $\epsilon \to 0$ to conclude proof of the lemma.
\end{proof}
        We now complete the proof of Theorem \ref{theorem}.
\begin{proof}[Proof of Theorem \ref{theorem}]
    It remains to consider $\frac{2d}{d-2} < p < \infty$ and $p = \infty$ for $d = 3$.
    
    It suffices to work with $t > 0$ as $t < 0$ will follow from time-reversal symmetry. By the density of Schwartz functions in $\dot{H}^1 \cap L^{p'}$, it suffices to consider Schwartz solutions of \eqref{NLS}. 
    
    For $0 < T \leq \infty$, we define the norm
    \begin{equation*}
        \|u\|_{X(T)} = \sup_{t \in [0,T)} |t|^{d(\frac{1}{2} - \frac{1}{p})} \|u(t)\|_{L^p_x}.
    \end{equation*}
    It then suffices to show
    \begin{equation}\label{bp/conclusion}
        \|u\|_{X(\infty)} \leq C\big(\|u_0\|_{\dot{H}^1}\big)\|u_0\|_{L^{p'}},
    \end{equation}
    for which we proceed with a bootstrap argument.

    Fix a small parameter $\eta > 0$ to be chosen later based on absolute constants and $\|u_0\|_{\dot{H}^1}$. A quick calculation shows that
    \begin{align*}
        (q,r) = \big(\tfrac{8p}{(d-2)[p(4-d) + d]}, \tfrac{4pd}{(d-2)(pd - 2p - d) + 4p}\big)
    \end{align*}
    is a non-endpoint Schr\"odinger-admissible pair for $p,d$ as in Theorem \ref{theorem}. 
    By Proposition \ref{spacetime bounds}, we may then decompose $[0,\infty)$ into $J = J(\eta, \|u_0\|_{\dot{H}^1})$ many intervals $I_j = [T_{j-1},T_j)$ on which
    \begin{equation}\label{bp/smallness}
        \|\nabla u\|_{L_t^{q,2}L_x^{r,2}(I_j)} < \eta.
    \end{equation}
    
    We aim to show that for all $ j = 1,\dots, J$,
    \begin{equation}\label{bp/bootstrap}
        \|u\|_{X(T_j)} \lesssim C(\|u_0\|_{\dot{H}^1})\Big(\|u_0\|_{L^{p'}} + \|u\|_{X(T_{j-1})} + \eta^\frac{4}{d-2} \|u\|_{X(T_j)}\Big).
    \end{equation}
    Taking $\eta$ sufficiently small relative to the constants in \eqref{bp/bootstrap}, we could then iterate over $j=1,\dots,J(\|u_0\|_{\dot{H}^1})$ to yield \eqref{bp/conclusion} and complete the proof of Theorem \ref{theorem}.

    We therefore focus on \eqref{bp/bootstrap}. Combining the linear dispersive decay \eqref{intro/linear dispersive decay} with Lemmas \ref{d3/early time} and \ref{d4/early time}, the Duhamel formula implies that for all $j$,
    \begin{equation}\label{bp/duhamel}
        \|u\|_{X(T_j)} \lesssim C(\|u_0\|_{\dot{H}^1})\|u_0\|_{L^{p'}} + \bigg\|\int_{t/2}^t e^{i(t-s)\Delta}\big[|u|^2u\big](s)ds \bigg\|_{X(T_j)}.
    \end{equation}
    It thus remains to consider the late-time interval $[t/2,t)$.

    Consider first $\frac{2d}{d-2} \leq p < \infty$. We take a Sobolev embedding and apply the linear dispersive decay \eqref{intro/linear dispersive decay} to find that
    \begin{equation}\begin{split}\label{bp/Sobolev}
        \bigg\|\int_{t/2}^t e^{i(t-s)\Delta}\big[&|u|^\frac{4}{d-2}u\big](s)ds \bigg\|_{L^p} \\
        & \lesssim \int_{t/2}^t \Big\|e^{i(t-s)\Delta}|\nabla|^{1-\frac{d}{2p}} \big[|u|^\frac{4}{d-2}u\big](s) \Big\|_{L^{\frac{2pd}{2p+d}}_x} ds \\
        & \lesssim \int_{t/2}^t |t-s|^{\frac{d + 2p - pd}{2p}}\Big\||\nabla|^{1-\frac{d}{2p}} \big[|u|^\frac{4}{d-2}u\big](s) \Big\|_{L^{\frac{2pd}{2pd-2p-d}}_x} ds.
    \end{split}\end{equation}
    H\"older's inequality and Sobolev embedding then imply that
    \begin{align*}
        \bigg\|\int_{t/2}^t e^{i(t-s)\Delta}\big[|u|^\frac{4}{d-2}u\big](s)ds \bigg\|_{L^p} 
        & \lesssim \int_{t/2}^t |t-s|^{\frac{d + 2p - pd}{2p}}\|u(s)\|_{L_x^p}\|\nabla u(s)\|^\frac{4}{d-2}_{L_x^{r}} ds.
    \end{align*} 

    For $p = \infty$ when $d = 3$, we use Lemma \ref{d3/Linfty Sobolev} in place of Sobolev embedding. 
    Applying Lemma \ref{d3/Linfty Sobolev} and standard estimates, we then find similarly that
    \begin{align*}
        \bigg\|\int_{t/2}^t e^{i(t-s)\Delta}\big[|u|^4u\big](s)ds & \bigg\|_{L_x^\infty}
        \lesssim \int_{t/2}^t |t-s|^{-1/2}\Big\|\nabla\big[|u|^4u\big](s)\Big\|_{L_x^{3/2,1}} ds\\
        & \lesssim \int_{t/2}^t |t-s|^{-1/2}\|u(s)\|_{L_x^\infty} \|\nabla u(s)\|^4_{L_x^{12/5,4}} ds.
    \end{align*}
    
    In either case, the nesting of Lorentz spaces implies that we may estimate
    \begin{align*}
        \bigg\|\int_{t/2}^t e^{i(t-s)\Delta}\big[|u|^\frac{4}{d-2}u\big](s)ds \bigg\|_{L^p} 
        & \lesssim \int_{t/2}^t |t-s|^{\frac{d + 2p - pd}{2p}}\|u(s)\|_{L_x^p}\|\nabla u(s)\|^\frac{4}{d-2}_{L_x^{r,2}} ds.
    \end{align*} 
    We note that $|s| \sim |t|$ for $s \in [t/2, t)$. H\"older's inequality then implies that
    \begin{align}
        \bigg\|\int_{t/2}^t e^{i(t-s)\Delta}&\big[|u|^\frac{4}{d-2}u\big](s)ds \bigg\|_{L^p} \nonumber \\
        & \lesssim |t|^{-d(\frac{1}{2} - \frac{1}{p})}\int_{t/2}^t |t-s|^{\frac{d + 2p - pd}{2p}}\|u\|_{X(s)}\|\nabla u(s)\|^\frac{4}{d-2}_{L_x^{r,2}} ds. \label{d4/failure 1}\\
        & \sim |t|^{-d(\frac{1}{2} - \frac{1}{p})}\Big\|\|u(s)\|_{X(s)} \|\nabla u(s)\|^\frac{4}{d-2}_{L^{r,2}_x}\Big\|_{L_s^{\frac{2p}{p(d-4)+d},1}([0,t))}.  \label{d4/failure 2}
    \end{align}
    We note that for $d = 4$, this requires $p < \infty$ as otherwise $L_s^{\frac{2p}{p(d-4)+d},1}$ is trivial.
    
    We now take the supremum over $t \in [0,T_j)$ and decompose $[t/2,t)$ into $[t/2,t) \cap I_j$ and $[t/2,t) \cap [0,T_{j-1})$. As $(q,r)$ is a non-endpoint Schr\"odinger-admissible pair and $d \in \{3,4\}$, Proposition \ref{spacetime bounds} and \eqref{bp/smallness} then imply
    \begin{equation}\begin{split}\label{bp/1}
        \bigg\|\int_{t/2}^t e^{i(t-s)\Delta} & \big[|u|^\frac{4}{d-2}u\big](s)ds \bigg\|_{X(T_j)} \\
        & \lesssim \|u\|_{X(T_{j-1})}\|\nabla u\|^\frac{4}{d-2}_{L_t^{q,\frac{4}{d-2}}L_x^{r,2}} + \|u\|_{X(T_j)}\|\nabla u\|^\frac{4}{d-2}_{L_t^{q,\frac{4}{d-2}}L_x^{r,2}(I_j)} \\
        & \lesssim C(\|u_0\|_{\dot{H}^1}) \bigg(\|u\|_{X(T_{j-1})} + \eta^\frac{4}{d-2} \|u\|_{X(T_j)}\bigg).
    \end{split}\end{equation}
    
    Combining estimates \eqref{bp/duhamel} and \eqref{bp/1} then yields the bootstrap statement \eqref{bp/bootstrap}. With earlier considerations, taking $\eta$ sufficiently small and iterating over $j = 1, \dots, J(\|u_0\|_{\dot{H}^1})$ then concludes the proof of Theorem \ref{theorem}.
\end{proof}

\section{Besov theory}\label{besov}
    For the proof of Theorem \ref{edge case theorem}, it will be necessary to decompose our solutions into Littlewood-Paley pieces and employ the use of Besov spaces. Here we recall the definition of homogeneous Besov spaces and briefly develop some Besov theory for \eqref{NLS}. For a more general treatment of nonlinear Schr\"odinger equations in Besov spaces, see \cite{besov-strichartz} and references therein. For textbook treatments of Besov spaces, we direct the interested reader to \cite{grafakos-modern,triebel}, though we caution that many conventions exist for the notation.
\begin{definition}
    Fix $d \geq 1$, $1 \leq p,q \leq \infty$, and $s \in \R$. The homogeneous Besov space $\dot{B}^s_{p,q}$ is the completion of the Schwartz functions $\mathcal{S}(\R)$ with respect to the norm
    \begin{equation*}
        \|f\|_{\dot{B}^s_{p,q}} = \big\| N^s\|f_N(x)\|_{L^p_x(\R^d)}\big\|_{\ell_N^q(2^\Z)}.
    \end{equation*}
    We note that Bernstein's inequality \eqref{intro/bernstein} implies that the factor $N^s$ acts as $|\nabla|^s$.
\end{definition}

Throughout our analysis, it will be crucial to employ a simple paraproduct decomposition for the cubic nonlinearity of \eqref{NLS} in $d = 4$. This will allow us to restrict attention to terms which feature a high frequency. We recall such a decomposition in the following lemma.
\begin{lemma}[Paraproduct decomposition]\label{besov/paraproduct}
    Suppose that $f^1,f^2,f^3 \in L^2$. Then $P_N(f^1f^2f^3)$ can be expressed as
    \begin{equation*}\begin{split}\label{besov/paraproduct general}
        P_N(f^1f^2f^3) 
        & = P_N\bigg(f^1_{\geq N/8} f^2 f^3+ f^1_{< N/8} f^2_{\geq N/8} f^3 + f^1_{< N/8} f^2_{< N/8} f^3_{\geq N/8}\bigg) \\
        & = P_N \sum_{N_1 \gtrsim N}\sum_{(i,j,k)}  \schem(f^i_{N_1} f^j f^k),
    \end{split}\end{equation*}
    where $(i,j,k)$ is summed over all permutations of $(1,2,3)$ and $\schem$ is defined in \eqref{notation/schem}.
    In particular,
    \begin{equation*}\begin{split}\label{besov/paraproduct simple}
        |P_N(f^1f^2f^3)| 
        & \lesssim \sum_{(i,j,k)}\sum_{N_1 \gtrsim N}M\big[f^i_{N_1}\cdot Mf^j \cdot Mf^k\big],
    \end{split}\end{equation*}
    uniformly for $N \in 2^\Z$ where $M$ is the Hardy--Littlewood maximal function.
\end{lemma}

In a similar vein to Proposition \ref{spacetime bounds}, we may establish mixed Besov-type spacetime bounds for solutions to \eqref{NLS} with a initial data in $\dot{B}^1_{2,1}$.
\begin{proposition}[Besov spacetime bounds]\label{besov/ell1}
    Let $(p,q)$ be a Schr\"odinger-admissible pair for $d = 4$. Suppose that $u_0 \in \dot{B}^1_{2,1}(\R^4) \subset \dot{H}^1(\R^4)$ satisfies the hypotheses of Theorem \ref{well-posedness}. Then the corresponding global solution $u(t)$ to \eqref{NLS} with initial data $u_0$ satisfies
    \begin{equation*}
        \sum_{N \in 2^\Z} N \|u_N\|_{L_t^p L_x^q} \leq C\big(\|u_0\|_{\dot{B}^1_{2,1}}\big).
    \end{equation*}
    The same estimates hold for the final-state problem with $u_0$ replaced by $u_\pm$.
\end{proposition}
\begin{proof}
    We focus on the initial-value problem first before remarking on the needed changes for the final-state problem. We first consider $p = q = 3$ before generalizing to all Schr\"odinger-admissible pairs $p,q$. 
    
    As $\dot{B}^1_{2,1} \hookrightarrow H^2$, it suffices to consider $u_0 \in H^2$ so that persistence of regularity (see \cite{focusing-d4,defocusing-d4}) implies that all mixed Besov norms considered are finite.

    We proceed via a bootstrap argument and introduce a small parameter $\eta > 0$ to be chosen later based on absolute constants. By Theorem \ref{spacetime bounds}, we may decompose $\R$ into $J = J(\eta, \|u_0\|_{\dot{B}^1_{2,1}})$ many intervals $I_j$ on which 
    \begin{equation}\label{besov/bounds/smallness}
        \|u\|_{L_{t,x}^6(I_j \times \R^4)} < \eta.
    \end{equation}
    
    By the Duhamel formula and Strichartz estimates, for each spacetime slab $I_j \times \R^4$ and any $t_j \in I_j$, we may estimate
    \begin{align*}
        \sum_{N \in 2^\Z} N \|u_N\|_{L^3_{t,x}(I_j)} & \lesssim \|u(t_j)\|_{\dot{B}^1_{2,1}} + \sum_{N \in 2^\Z} N \big\|(|u|^2 u)_N  \big\|_{L_{t,x}^{3/2}(I_j)}.
    \end{align*}
    For the second term, we apply the paraproduct decomposition, Lemma \ref{besov/paraproduct}. H\"older's inequality and the boundedness of the Hardy-Littlewood maximal function then imply that
    \begin{align*}
        \sum_{N \in 2^\Z} N \|u_N\|_{L^3_{t,x}(I_j)} & \lesssim \|u(t_j)\|_{\dot{B}^1_{2,1}} + \sum_{N_1 \gtrsim N} N \|u_{N_1} (Mu)^2\|_{L^{3/2}_{t,x}(I_j)} \\
        & \lesssim \|u(t_j)\|_{\dot{B}^1_{2,1}} +  \|u\|^2_{L^6_{t,x}(I_j)}\sum_{N_1 \gtrsim N} N \|u_{N_1}\|_{L^3_{t,x}(I_j)}.
    \end{align*}
    
    Summing over $N$ first, \eqref{besov/bounds/smallness} then implies that
    \begin{equation}\begin{split}\label{besov/1}
        \sum_{N \in 2^\Z} N \|u_N\|_{L^3_{t,x}(I_j)} 
        & \lesssim \|u(t_j)\|_{\dot{B}^1_{2,1}} +  \eta^2 \sum_{N_1} N_1 \|u_{N_1}\|_{L^3_{t,x}(I_j)}.
    \end{split}\end{equation}
    Choosing $\eta$ sufficiently small relative to the constants in \eqref{besov/1}, a standard bootstrap argument yields
    \begin{equation*}
        \sum_{N\in 2^\Z} N \|u_N\|_{L_{t,x}^3(I_j)} \leq C\big(\|u_0\|_{\dot{B}^1_{2,1}}\big)\|u(t_j)\|_{\dot{B}^1_{2,1}}.
    \end{equation*}
    
    By the Duhamel formula and Strichartz estimates, for any Schr\"odinger admissible pair $(p,q)$ we find similarly that
    \begin{align*}
        \sum_{N \in 2^\Z} N \|u_N\|_{L_t^p L_x^q(I_j)} & \lesssim \|u(t_j)\|_{\dot{B}^1_{2,1}} +  \|u\|^2_{L^6_{t,x}} \sum_{N_1} N_1 \|u_{N_1}\|_{L^3_{t,x}} \\
        & \leq C\big(\|u_0\|_{\dot{B}^1_{2,1}}\big)\|u(t_j)\|_{\dot{B}^1_{2,1}}.
    \end{align*}
    By the $p = \infty, q = 2$ estimate, an iterative argument with appropriately chosen $t_j$ then implies that
    \begin{align*}
        \sum_{N \in 2^\Z} N \|u_N\|_{L_t^p L_x^q(I_j)} & \leq C\big(\|u_0\|_{\dot{B}^1_{2,1}}, j\big).
    \end{align*}
    Summing over $j = 1, \dots, J(\|u_0\|_{\dot{B}^1_{2,1}})$, this concludes the proof of the proposition for the initial-value problem. 
    
    To adapt the preceding proof to the final-state problem, it suffices to change every instance of $u_0$ to $u_\pm$ and one instance of $u(t_j)$ to $u_\pm$.
\end{proof}

In addition to spacetime bounds, we will require a stability result for initial data in $\dot{B}^1_{2,1}(\R^4)$. This proof parallels the usual proof of stability in $\dot{H}^1$, only with a paraproduct decomposition used to understand the nonlinearity of \eqref{NLS}.
\begin{proposition}[$\dot{B}^1_{2,1}$ stability]\label{besov/stability}
    Fix $d = 4$ and let $p,q$ be a Schr\"odinger-admissible pair. Suppose that $u_0, v_0$ satisfy the hypotheses of Theorem \ref{well-posedness} and obey the bound $\|u_0\|_{\dot{B}^1_{2,1}}, \|v_0\|_{\dot{B}^1_{2,1}} \leq R$. Let $u(t), v(t)$ be the corresponding solutions to \eqref{NLS} with initial data $u_0,v_0$ respectively. Then
    \begin{equation*}
        \sum_{N \in 2^\Z}\|\nabla (u_N - v_N)\|_{L_t^p L_x^q} \leq C(R) \|u_0 - v_0\|_{\dot{B}^1_{2,1}}.
    \end{equation*}
    The same estimates hold for the final-state problem with $u_0,v_0$ replaced by $u_\pm, v_\pm$. 
\end{proposition}
\begin{proof}
    We focus on the initial-value problem first before remarking on the needed changes for the final-state problem. We first consider $p = q = 3$ before generalizing to all Schr\"odinger-admissible pairs $p,q$.
    
    We proceed via a bootstrap argument and introduce a small parameter $\eta > 0$ to be chosen later based on absolute constants and $R$. By Proposition \ref{besov/ell1}, we may decompose $\R$ into $J = J(R,  \eta)$ many intervals $I_j$ on which 
    \begin{equation}\label{besov/stability/smallness}
        \big\||u|\wedge|v|\big\|_{L_{t,x}^6(I_j)} \leq \|u\|_{L_{t,x}^6(I_j)} + \|v\|_{L_{t,x}^6(I_j)} < \eta.
    \end{equation}
    
    By the Duhamel formula and Strichartz estimates, for each spacetime slab $I_j \times \R^4$ and any $t_j \in I_j$, we may estimate
    \begin{equation}\label{besov/stability/1.5}\begin{split}
        \sum_N N \|P_N (u - v)\|_{L_{t,x}^{3}} & \lesssim \|u(t_j) - v(t_j)\|_{\dot{B}^1_{2,1}} + \sum_N N \Big\|P_N \big(|u|^2 u - |v|^2 v\big) \Big\|_{L_{t,x}^{3/2}} \\
        & = \|u(t_j) - v(t_j)\|_{\dot{B}^1_{2,1}} + \rn{2}.
    \end{split}\end{equation}
    
    We focus on the second term and decompose the nonlinearity as
    \begin{align*}
        |u|^2 u - |v|^2 v & = u^2\overline{(u - v)} + \overline{v} u (u-v) + |v|^2 (u-v).
    \end{align*}
    Then
    \begin{align*}
        \rn{2} & \leq \sum_N N \bigg\{\Big\| P_N \big[u^2\overline{(u-v)}\big] \Big\|_{L_{t,x}^{3/2}} + \Big\| P_N \big[\overline{v} u(u-v)\big] \Big\|_{L_{t,x}^{3/2}} + \Big\| P_N \big[|v|^2(u-v)\big] \Big\|_{L_{t,x}^{3/2}}\bigg\}.
    \end{align*}
    
    Applying the paraproduct decomposition, Lemma \eqref{besov/paraproduct}, we schematically have three terms, each term corresponding to the case where one of $(u-v),u,v$ lies at the high frequency $N_1$. Expanding the sum and bounding $u,v$ by $|u|\wedge |v|$ for convenience, we then find that
    \begin{align*}
        \rn{2} \lesssim &\; \sum_{N_1 \gtrsim N} N\bigg[\big\|u_{N_1} M(|u|\wedge |v|)M(u-v) \big\|_{L_{t,x}^{3/2}} + \big\|v_{N_1}M\big(|u|\wedge|v|\big)M(u-v) \big\|_{L_{t,x}^{3/2}}\bigg] \\
        & \; + \sum_{N_1 \gtrsim N} N \big\|(u-v)_{N_1} \big[M\big(|u|\wedge|v|\big)\big]^2 \big\|_{L_{t,x}^{3/2}}.
    \end{align*}
    
    Summing in $N$ first and then applying H\"older's inequality, we may estimate
    \begin{equation*}\begin{split}
        \rn{2} \lesssim & \; \sum_{N_1} N_1\bigg[\big\|u_{N_1} M\big(|u|\wedge|v|\big)M(u-v) \big\|_{L_{t,x}^{3/2}} +\big\|v_{N_1}M\big(|u|\wedge|v|\big)M(u-v) \big\|_{L_{t,x}^{3/2}}\bigg] \\
        & + \sum_{N_1} N_1 \big\|(u-v)_{N_1} \big[M\big(|u|\wedge|v|\big)\big]^2 \big\|_{L_{t,x}^{3/2}} \\
        \lesssim & \;\big\||u|\wedge|v|\big\|_{L_{t,x}^6}\|u-v\|_{L_t^3 L_x^{12}}\bigg(\sum_{N_1} N_1 \|u_{N_1}\|_{L_t^6 L_x^{12/5}} + \sum_{N_1} {N_1} \|v_{N_1}\|_{L_t^6 L_x^{12/5}}\bigg) \\
        & + \big\||u|\wedge|v|\big\|^2_{L_{t,x}^6}\sum_{N_1} N_1 \|(u-v)_{N_1}\|_{L_{t,x}^3}.
    \end{split}\end{equation*}
    Proposition \ref{besov/ell1} and \eqref{besov/stability/smallness} then imply that 
    \begin{equation}\begin{split}\label{besov/stability/2}
        \rn{2} \lesssim C(R) \eta\sum_{N_1} N_1 \|(u-v)_{N_1}\|_{L_{t,x}^3}.
    \end{split}\end{equation}
    
    Combining \eqref{besov/stability/1.5} and \eqref{besov/stability/2}, we then find that for each spacetime slab $I_j \times \R^4$,
    \begin{equation}\label{besov/stability/3}
        \sum_N N \|P_N (u - v)\|_{L_{t,x}^{3}} \lesssim \|u(t_j) - v(t_j)\|_{\dot{B}^1_{2,1}} + C(R) \eta\sum_{N_1} N_1 \|(u-v)_{N_1}\|_{L_{t,x}^3}.
    \end{equation}
    Taking $\eta$ sufficiently small based only on the constants in \eqref{besov/stability/3} and $R$, a bootstrap argument then implies that for all $j = 1,\dots, J(R)$,
    \begin{equation}
        \sum_N N \|(u - v)_N\|_{L_{t,x}^{3}(I_J)} \leq C(R)\|u(t_j) - v(t_j)\|_{\dot{B}^1_{2,1}}.
    \end{equation}
    For any Schr\"odinger-admissible pair, the Duhamel formula and Strichartz inequalities similarly imply that
    \begin{align*}
        \sum_N N \|(u - v)_N\|_{L_t^p L_x^q(I_j)} 
        & \leq C(R)\|u(t_j) - v(t_j)\|_{\dot{B}^1_{2,1}}.
    \end{align*}
    By the $p = \infty, q = 2$ estimate, an iterative argument with appropriately chosen $t_j$ then implies that
    \begin{align*}
        \sum_N N \|(u - v)_N\|_{L_t^p L_x^q(I_j)} & \leq C(R, j)\|u_0 - v_0\|_{\dot{B}^1_{2,1}}.
    \end{align*}
    Summing over $j = 1, \dots, J(R)$, this concludes the proof of the proposition for the initial-value problem.
    
    
    To adapt the preceding proof to the final-state problem, it is sufficient to change every instance of $u_0,v_0$ to $u_\pm,v_\pm$ respectively and one instance of $u(t_j),v(t_j)$ to $u_\pm, v_\pm$.
\end{proof}

\section{Proof of Theorem \ref{edge case theorem}}\label{edge}
    In this section, we prove Theorem \ref{edge case theorem}. Though we again use a bootstrap argument, we can no longer close the bootstrap by decomposing $\R$ into small intervals as was done in \eqref{integrable/smallness} and \eqref{bp/smallness}. This complication arises because our spacetime norms appear with $L_t^\infty$; see \eqref{edge/optimized}. Instead, we induct on the size of the initial data and close the bootstrap by considering small perturbations of the initial data.
\begin{proof}[Proof of Theorem \ref{edge case theorem}]
    For $f : \R \times \R^4 \to \C$, we define the norm
    \begin{equation*}
        \|f(t,x)\|_{X} = \sup_{t \neq 0} |t|^2 \|f(t)\|_{L_x^\infty}.
    \end{equation*}
    By the density of Schwartz functions in $\dot{B}^1_{2,1} \cap L^{p'}$, it suffices to consider Schwartz solutions of \eqref{NLS}.
    
    We proceed via induction on $\|u_0\|_{\dot{B}^1_{2,1}}$. As Theorem \ref{edge case theorem} holds trivially for $\|u_0\|_{\dot{B}^1_{2,1}} = 0$, it suffices to show the inductive step. 
    
    Suppose for the sake of induction that there exists some $R_0 \geq 0$ and $C(R_0)$ such that for all $\|u_0\|_{\dot{B}^1_{2,1}} \leq R_0$, the corresponding solution $u(t)$ to \eqref{NLS} with initial data $u_0$ satisfies
    \begin{equation}\label{edge/decay}
        \|u\|_{X} \leq C(R_0) \|u_0\|_{L^1}.
    \end{equation}
    By induction, it then suffices to show that \eqref{edge/decay} extends to all $\|u_0\|_{\dot{B}^1_{2,1}} \leq R_0 + 1$, perhaps with a new constant $C(R_0)$.

    We show this incrementally. Fix some $\epsilon > 0$ sufficiently small to be chosen later based only on $R_0$ and absolute constants. 
    Suppose for the sake of iteration that there exists $k$ with $R_0 + (k+1)\epsilon \leq R_0 + 1$ such that for all $\|v_0\|_{\dot{B}^1_{2,1}} \leq R_0 + k\epsilon$, the corresponding solution $v(t)$ to \eqref{NLS} with initial data $v_0$ satisfies
    \begin{equation}\label{edge/inductive assumption}
        \|v(t)\|_{X} \leq C(R_0,k,\epsilon)\|v_0\|_{L^1}.
    \end{equation}
    We then aim to show that for all $\|u_0\|_{\dot{B}^1_{2,1}} \leq R_0 + (k+1)\epsilon$, the corresponding solution $u(t)$ to \eqref{NLS} with initial data $u_0$ satisfies
    \begin{equation}\label{edge/inductive conclusion}
        \|u(t)\|_{X} \leq C(R_0,k,\epsilon)\|u_0\|_{L^1}.
    \end{equation}
    Provided that $\epsilon = \epsilon(R_0)$ is chosen based only on $R_0$, iterating over $k = 0,\dots, \epsilon^{-1} - 1$ will then extend \eqref{edge/decay} to all $\|u_0\|_{\dot{B}^1_{2,1}} \leq R_0 + 1$, potentially with a new constant $C(R_0)$. 
    
    We therefore focus on \eqref{edge/inductive conclusion} and proceed via a bootstrap argument. Suppose that $u_0$ satisfies the hypotheses of Theorem \ref{well-posedness} with $\|u_0\|_{\dot{B}^1_{2,1}} \leq R_0 + (k+1)\epsilon$. To make use of the iterative assumption \eqref{edge/inductive assumption}, we decompose $u_0$ as 
    \begin{equation*}
        u_0 = \big(\tfrac{R_0+ k\epsilon}{R_0+ (k+1)\epsilon}\big) u_0 + \big(\tfrac{\epsilon}{R_0+(k+1)\epsilon}\big) u_0 = v_0 + w_0.
    \end{equation*}
    Then $\|v_0\|_{\dot{B}^1_{2,1}} \leq R_0+ k\epsilon$; $\|w_0\|_{\dot{B}^1_{2,1}} \leq \epsilon$; and $\|v_0\|_{L^1},\|w_0\|_{L^1} \leq \|u_0\|_{L^1}$. 
    
    Let $v(t)$ be the solution to \eqref{NLS} with initial data $v_0$ and let $w(t) = u(t) - v(t)$. Note that Proposition \ref{besov/stability} then implies that
    \begin{equation}\label{edge/w smallness}
        \sum_N N\|w_N\|_{L_t^\infty L_x^2} \leq C(R_0) \epsilon,
    \end{equation}
    because $R_0+ (k+1)\epsilon \leq R_0 + 1$.
    
    As $\|v_0\|_{\dot{B}^1_{2,1}} \leq R_0+ k\epsilon$, the iterative assumption \eqref{edge/inductive assumption} implies that
    \begin{equation}\label{edge/bootstrap u}
        \|u\|_{X} \leq \|v\|_{X} + \|w\|_{X} \leq C(R_0,k,\epsilon)\|u_0\|_{L^1} + \|w\|_{X}.
    \end{equation}
    To prove \eqref{edge/inductive conclusion}, we then aim to show that $w$ satisfies the bootstrap statement
    \begin{equation}\label{edge/bootstrap w}
        \|w\|_{X} \leq C(R_0,k,\epsilon)\|u_0\|_{L^1} + C(R_0)\epsilon \|w\|_{X}.
    \end{equation}
    Taking $\epsilon \leq 1/2C(R_0)$ would then imply
    \begin{equation*}
        \|w\|_{X} \leq C(R_0, k, \epsilon) \|u_0\|_{L^1}.
    \end{equation*}
    Along with \eqref{edge/bootstrap u}, this would imply the inductive statement \eqref{edge/inductive conclusion} and conclude the proof of the theorem.

    We thus focus our attention on \eqref{edge/bootstrap w}. By definition, $w(t) = u(t) - v(t)$ solves the coupled equation
    \begin{align*}
        i w_t + \Delta w \pm (|u|^2u - |v|^2v) = 0,
    \end{align*}
    in the strong sense. Expanding $u = w + v$, we note that the nonlinearity schematically has 3 terms which we express as
    \begin{equation*}
        i w_t + \Delta w + \schem(w^3) + \schem(w^2 v) + \schem(w v^2) = 0.
    \end{equation*}
    Decomposing into early-\ and late-time intervals, the Duhamel formula for $w$ can then be written as
    \begin{equation}\label{edge/duhamel w}\begin{split}
        w(t) & = e^{it\Delta} w_0 \mp i\int_0^{t/2} e^{i(t-s)\Delta}\big[|u|^2 u - |v|^2 v\big](s) ds \\
        & \hspace{11pt} + \int_{t/2}^t e^{i(t-s)\Delta}\big[\schem(w^3) + \schem(w^2 v) + \schem(w v^2)\big](s) ds \\
        & = \rn{1} + \rn{2} + \rn{3}(w^3) + \rn{3}(w^2v) + \rn{3}(wv^2).
    \end{split}\end{equation}

    For term $\rn{1}$ in \eqref{edge/duhamel w}, the linear dispersive decay \eqref{intro/linear dispersive decay} immediately implies that
    \begin{equation}\label{edge/I}
        \|\rn{1}\|_{X} \lesssim  \|w_0\|_{L^1} \leq \|u_0\|_{L^1}
    \end{equation}
    which is acceptable for the bootstrap statement \eqref{edge/bootstrap w}. For term $\rn{2}$ in \eqref{edge/duhamel w}, because $u$ and $v$ both satisfy \eqref{NLS} and $\|u_0\|_{\dot{B}^1_{2,1}}, \|v_0\|_{\dot{B}^1_{2,1}} \leq R_0 + 1$, Lemma \ref{d4/early time} implies that
    \begin{equation}\begin{split}\label{edge/II}
        \|\rn{2}\|_{X} & = \bigg\| \int_0^{t/2} e^{i(t-s)\Delta}\big[|u|^2 u - |v|^2 v\big](s) ds\bigg\|_{X}
        \leq C(R_0) \|u_0\|_{L^1},
    \end{split}\end{equation}
    which is similarly acceptable for \eqref{edge/bootstrap w}. 
    
    It then remains to estimate the terms $\rn{3}(w^3)$, $\rn{3}(w^2v)$, and $\rn{3}(w v^2)$ in \eqref{edge/duhamel w}.
    We first consider general terms of the form $\rn{3}(*)$ before specializing.
    To align with the notation of Lemma \ref{besov/paraproduct}, we then consider terms of the form
    \begin{equation}\label{edge/general term}
        \sum_{N_1 \gtrsim N} \big\|\rn{3}\big(P_N[f_{N_1} g h]\big)\big\|_{X} = \sum_{N_1 \gtrsim N} \bigg\|\int_{t/2}^t e^{i(t-s)\Delta}P_N[f_{N_1}gh](s) ds\bigg\|_X,
    \end{equation}
    for arbitrary functions $f,g,h$. Before summing over $N_1 \gtrsim N$, we first consider individual terms $\rn{3}\big( P_N[f_{N_1} g h]\big)$.
    
    We decompose $g$ into Littlewood-Paley pieces and then introduce an integration cutoff $B > 0$ as
    \begin{align*}
        \big\| \rn{3}\big(P_N [f_{N_1}gh]\big)\big\|_{L_x^\infty} & \lesssim \sum_{N_2} \bigg\|\int_{t/2}^t e^{i(t-s)\Delta} P_N[f_{N_1} g_{N_2} h](s) ds \bigg\|_{L_x^\infty} \\
        & = \sum_{N_2}\int_{t/2}^{(t-B)\wedge t/2} \big\|e^{i(t-s)\Delta} P_N[f_{N_1} g_{N_2} h](s)\big\|_{L_x^\infty} ds \\ 
        & \hspace{11pt}
        + \sum_{N_2} \int_{(t-B)\wedge t/2}^t \big\|e^{i(t-s)\Delta} P_N[f_{N_1} g_{N_2}h](s)\big\|_{L_x^\infty} ds.
    \end{align*}
    In doing so, we have isolated the singularity at $s = t$ into the interval $((t-B)\wedge t/2, t]$.
    
    On the interval $(t/2, (t-B) \wedge t/2]$, away from the singularity, we apply the linear dispersive decay \eqref{intro/linear dispersive decay} directly. On the interval $((t-B)\wedge t/2, t]$, we apply Bernstein's inequality \eqref{intro/bernstein} and then the conservation of mass \eqref{intro/conservation of mass}. Because $P_N$ is bounded on $L^p$, this implies that
    \begin{align*}
        \big\| \rn{3}\big(P_N [f_{N_1}gh]\big)\big\|_{L_x^\infty}
        & \lesssim \sum_{N_2} \int_{t/2}^{(t-B)\wedge t/2}  |t-s|^{-2}\|[f_{N_1} g_{N_2} h](s)\|_{L_x^1} ds \\
        & \hspace{11pt} + \sum_{N_2} \int_{(t-B)\wedge t/2}^t N^2\big\|[f_{N_1} g_{N_2} h](s)\big\|_{L_x^2} ds.
    \end{align*}
    
    For each term, we place $h$ into the bootstrap norm and note that $|s| \sim |t|$ for $s \in [t/2, t)$. Doing so, we find that
    \begin{align*}
        \big\| \rn{3}\big(P_N [f_{N_1}gh]\big)\big\|_{L_x^\infty}
        & \lesssim |t|^{-2} \|h\|_X\sum_{N_2} \int_{t/2}^{(t-B)\wedge t/2}  |t-s|^{-2}\|[f_{N_1} g_{N_2}](s)\|_{L_x^1} ds \\
        & \hspace{11pt} + |t|^{-2} \|h\|_X \sum_{N_2} \int_{(t-B)\wedge t/2}^t N^2\big\|[f_{N_1} g_{N_2}](s)\big\|_{L_x^2} ds.
    \end{align*}
    H\"older's inequality and Bernstein's inequality \eqref{intro/bernstein} then imply that
    \begin{align*}
        \big\| \rn{3}(P_N & [f_{N_1}gh])\big\|_{L_x^\infty}\\
        & \lesssim |t|^{-2} \|h\|_{X} \sum_{N_2} \bigg[B^{-1} \|f_{N_1}\|_{L_t^\infty L_x^2}\|g_{N_2}\|_{L_t^\infty L_x^2} + B N^2\|f_{N_1}\|_{L_t^\infty L_x^2} \|g_{N_2}\|_{L_{t,x}^\infty}\bigg] \\
        & \lesssim |t|^{-2}\|f_{N_1}\|_{L_t^\infty L_x^2}\|h\|_{X}\sum_{N_2} N_2\|g_{N_2}\|_{L_t^\infty L_x^2}\big[B^{-1}N_2^{-1}  + B N^2N_2\big].
    \end{align*}
    
    Optimizing $B \sim N^{-1} N_2^{-1}$, we may then estimate
    \begin{equation*}
        \big\| \rn{3}\big(P_N [f_{N_1}gh]\big)\big\|_{L_x^\infty} \lesssim |t|^{-2}N\|f_{N_1}\|_{L_s^\infty L_x^2}\bigg(\sum_{N_2}N_2\|g_{N_2}\|_{L_s^\infty L_x^2}\bigg)\|h\|_{X}.
    \end{equation*}
    Returning our attention to \eqref{edge/general term}, we sum over $N_1 \gtrsim N$ to find
    \begin{equation}\begin{split}\label{edge/optimized}
        \sum_{N_1 \gtrsim N} \big\| \rn{3}&\big(P_N [f_{N_1}gh]\big)\big\|_{X} 
         \lesssim \bigg(\sum_NN\|f_{N}\|_{L_t^\infty L_x^2}\bigg)\bigg(\sum_N N\|g_{N}\|_{L_t^\infty L_x^2}\bigg)\|h\|_{X},
    \end{split}\end{equation}
    which will be sufficient for our analysis of $\rn{3}(w^3)$,  $\rn{3}(w^2v)$, and $\rn{3}(wv^2)$.
    

    We turn our attention to $\rn{3}(w^2v)$. Applying the paraproduct decomposition, Lemma \ref{besov/paraproduct}, and then \eqref{edge/optimized}, we find that
    \begin{align*}
        \sum_{N} \big\|P_N \rn{3}(w^2 v)\big\|_{X} & \lesssim \sum_{N_1 \gtrsim N} \big\|\rn{3}\big(P_N[v_{N_1} w^2]\big)\big\|_{X} + \sum_{N_1 \gtrsim N} \big\|\rn{3}\big(P_N[w_{N_1} wv]\big)\big\|_{X} \\
        & \lesssim \|w\|_{X} \bigg(\sum_N N\|w_{N}\|_{L_t^\infty L_x^2}\bigg)\bigg(\sum_N N\|v_{N}\|_{L_t^\infty L_x^2}\bigg).
    \end{align*}
    With Proposition \ref{besov/ell1} and \eqref{edge/w smallness}, this implies that
    \begin{align}\label{edge/w 1}
        \sum_{N} \big\|P_N \rn{3}(w^2 v)\big\|_{X} & \lesssim C(R_0) \epsilon \|w\|_{X}.
    \end{align}
    
    Repeating this argument for $\rn{3}(w^3)$, we find similarly that
    \begin{equation}\begin{split}\label{edge/w 2}
        \sum_{N} \big\|P_N \rn{3}(w^3)\big\|_{X} & \lesssim \bigg(\sum_N N\|w_N\|_{L_t^\infty L_x^2} \bigg)^2 \|w\|_{X} \\
        & \lesssim C(R_0) \epsilon \|w\|_{X}.
    \end{split}\end{equation}

    We finally turn our attention to $\rn{3}(w v^2)$. Applying the paraproduct decomposition, Lemma \ref{besov/paraproduct}, and then \eqref{edge/optimized}, we find that
    \begin{align*}
        \sum_{N} \big\|P_N \rn{3}(w v^2)\big\|_{X} & \lesssim \sum_{N_1 \gtrsim N} \big\|\rn{3}\big(P_N[w_{N_1} v^2]\big) \big\|_{L_x^\infty} + \sum_{N_1 \gtrsim N} \big\|\rn{3}\big(P_N[v_{N_1} wv]\big) \big\|_{L_x^\infty} \\
        & \lesssim \bigg(\sum_N N\|w_N\|_{L_t^\infty L_x^2}\bigg)\bigg(\sum_N N\|v_N\|_{L_t^\infty L_x^2}\bigg)\|v\|_X.
    \end{align*}
    Because $\|v_0\|_{\dot{B}^1_{2,1}},\|w_0\|_{\dot{B}^1_{2,1}} \leq R_0 + k \epsilon \leq R_0 + 1$, we may apply \eqref{edge/w smallness}, the iterative assumption \eqref{edge/inductive assumption}, and Proposition \ref{besov/ell1} to estimate
    \begin{align}\label{edge/w 3}
        \sum_{N} \big\|P_N \rn{3}(w v^2)\big\|_{X} & \lesssim C(R_0,k,\epsilon) \|u_0\|_{L^1}.
    \end{align}

    Combining the expansion \eqref{edge/duhamel w} with the estimates \eqref{edge/I}, \eqref{edge/II}, \eqref{edge/w 1}, \eqref{edge/w 2}, and \eqref{edge/w 3} then implies the bootstrap statement \eqref{edge/bootstrap w}. By earlier considerations, this concludes the proof of Theorem \ref{edge case theorem}.
\end{proof}

We recall that the Besov spaces interpolate between Sobolev spaces of different regularity. In particular, for all $\alpha < 1 < \beta$,
\begin{equation*}
    \dot{B}^1_{2,1} = \big(\dot{H}^{\alpha}, \dot{H}^{\beta} \big)_{\frac{1 - \beta}{\alpha - \beta},1}.
\end{equation*}
Therefore, we gain an immediate corollary written in the standard Sobolev spaces:
\begin{corollary}\label{edge/corollary}
    Fix some $\alpha < 1 < \beta$. Given $u_0 \in L^1 \cap \dot{H}^{\alpha} \cap \dot{H}^{\beta}(\R^4)$ satisfying the hypotheses of Theorem \ref{well-posedness}, let $u(t)$ denote the unique global solution to \eqref{NLS} with initial data $u_0$. Then
    \begin{equation*}
        \|u(t)\|_{L_x^\infty} \leq C(\|u_0\|_{\dot{H}^{\alpha}},\|u_0\|_{\dot{H}^{\beta}})|t|^{-2} \|u_0\|_{L^1}.
    \end{equation*}
\end{corollary}

\section{Final-state problem} \label{scattering}
    In this section, we prove dispersive decay for the final-state problem, Theorems \ref{scattering theorem} and \ref{scattering edge case theorem}. We restrict attention to the scattering state $u_+$ as the case $u_-$ will follow from time-reversal symmetry. We then recall the Duhamel formula for the final-state problem with $u_+$ given:
\begin{equation}\label{scattering/duhamel}
    u(t) = e^{it\Delta} u_+ \pm i \int_t^\infty e^{i(t-s)\Delta} \big[|u|^{\frac{4}{d-2}}u\big](s) ds.
\end{equation}

The proofs of Theorems \ref{scattering theorem} and \ref{scattering edge case theorem} follow nearly identical arguments to the proofs of Theorems \ref{theorem} and \ref{edge case theorem} for the initial-value problem \eqref{NLS}. The only significant change arises from the decomposition $[0,t) = [0,t/2) \cup [t/2, t)$; see \eqref{integrable/early-late} for its use in the initial-value problem. This decomposition allowed us to estimate $|t-s| \gtrsim |t|$ for $s \in [0,t/2)$ and $|s| \gtrsim |t|$ for $s \in [t/2,t)$, but the exact form of the decomposition played no other part in the proof. To prove Theorems \ref{scattering theorem} and \ref{scattering edge case theorem}, it then suffices to find a similar decomposition of $[t,\infty)$ which allows for the same estimates.

For the case of $t < 0$, we now make the decomposition $[t,\infty) = [t,t/2) \cup [t/2, \infty)$. Then $|t-s| \gtrsim |t|$ for $s \in [t/2,\infty)$ and $|s| \gtrsim |t|$ for $s \in [t,t/2)$. With this decomposition, the proofs presented for Theorems \ref{theorem} and \ref{edge case theorem} can be adapted with only minor changes in notation.

It then remains to consider $t > 0$. In this case, we must modify the bootstrap argument in addition to the notation and decomposition. As these modifications will be consistent across all cases, we only present the proof in the case of $p = 3, d = 3$. This doubly serves to provide an example of explicit numbers for the proof of Theorem \ref{theorem} in the integrable case, see Section \ref{integrable}.
\begin{proof}[Proof of Theorems \ref{scattering theorem} and \ref{scattering edge case theorem}]
    As noted, we consider only the case of $t > 0$ for $u_+$ and we fix $p = 3$, $d = 3$ for concreteness. By the density of Schwartz functions in $\dot{H}^1 \cap L^{3/2}$, it suffices to consider Schwartz solutions of \eqref{NLS}.
    
    For $0 \leq T < \infty$, we define the norm
    \begin{equation*}
        \|u\|_{X(T)} = \sup_{t \geq T} |t|^{1/2}\|u(t)\|_{L_x^3}.
    \end{equation*}
    It then suffices to show
    \begin{equation}\label{scattering/conclusion}
        \|u\|_{X(0)} \leq C(\|u_+\|_{\dot{H}^1})\|u_+\|_{L^{p'}},
    \end{equation}
    for which we proceed with a bootstrap argument.

    Let $\eta > 0$ denote a small parameter to be chosen later, depending only on universal constants. Proposition \ref{spacetime bounds} then implies that we may decompose $[0,\infty)$ into $J = J(\|u_+\|_{\dot{H}^1},\eta)$ many intervals $I_j = [T_j,T_{j+1})$ on which
    \begin{equation}\label{scattering/smallness}
        \|u\|_{L_s^{8,4} L_x^{12}(I_j )} < \eta.
    \end{equation}
    
    We aim to show that for all $j = 1,\dots, J$,
    \begin{equation}\label{scattering/bootstrap}
        \|u\|_{X(T_j)} \lesssim \|u_0\|_{L^{3/2}} + C(\|u_+\|_{\dot{H^1}})\|u\|_{X(T_{j+1})} + \eta^4 \|u\|_{X(T_j)}.
    \end{equation}
    Choosing $\eta = \eta(R_0) > 0$ sufficiently small based on the constants in \eqref{scattering/bootstrap}, we could then iterate over $j = 1,\dots,J$ to yield \eqref{scattering/conclusion} and conclude the proof of Theorems \ref{scattering theorem} and \ref{scattering edge case theorem}.

    We therefore focus on \eqref{scattering/bootstrap}. Fix $t \in [T_j,\infty)$ and recall the Duhamel formula \eqref{scattering/duhamel}. By the linear dispersive decay \eqref{intro/linear dispersive decay}, the contribution of the linear term to $\|u(t)\|_{X(T_j)}$ is immediately seen to be acceptable:
    \begin{equation}\label{scattering/linear}
        \big\|e^{it\Delta} u_+\big\|_{X(T_j)} \lesssim \|u_+\|_{L_x^{3/2}}.
    \end{equation}
    We thus focus on the nonlinear correction. 
    
    By the linear dispersive decay \eqref{intro/linear dispersive decay} and H\"older's inequality, we may estimate
    \begin{align*}
        \bigg\|\int_t^\infty e^{i(t-s)\Delta}\big[|u|^4u\big](s)ds\bigg\|_{L_x^3}
        \lesssim \int_t^\infty |t-s|^{-1/2}\|u(s)\|_{L^3} \big\|u(s)\big\|^4_{L_x^{12}}ds.
    \end{align*}
    By definition, $|s|^{1/2}\|u(s)\|_{L^3} \leq \|u\|_{X(s)}$. Then
    \begin{align*}
        \bigg\| \int_t^\infty e^{i(t-s)\Delta} & \big[|u|^4u\big](s)ds\bigg\|_{L_x^p} \lesssim \int_t^\infty |t-s|^{-1/2}|s|^{-1/2}\|u\|_{X(s)}\|u(s)\|^4_{L_x^{12}}ds.
    \end{align*}
    
    We decompose $[t,\infty)$ into $[t,2t) $ and $ [2t,\infty) $. For $s \in [t,2t)$ we note that $|s| \sim |t|$, and for $s \in [2t,\infty)$, we note that $|t-s| \gtrsim |t|$. Then
    \begin{align*}
        \bigg\|\int_t^\infty e^{i(t-s)\Delta} \big[|u|^4 u\big](s) ds\bigg\|_{L_x^p} & \lesssim |t|^{-1/2}\int_t^{2t}|t-s|^{-1/2}\|u\|_{X(s)}\|u(s)\|^4_{L_x^{12}}ds\\
        &\hspace{11pt} + |t|^{-1/2}\int_{2t}^\infty |s|^{-1/2}\|u\|_{X(s)}\|u(s)\|^4_{L_x^{12}}ds.
    \end{align*}
    As $|t-s|^{-1/2}, |s|^{-1/2} \in L_s^{2,\infty}$, see \eqref{Lorentz/identity}, this decomposition and H\"older's inequality then imply
    \begin{align*}
        \bigg\|& \int_t^\infty e^{i(t-s)\Delta} \big[|u|^4u\big](s)ds\bigg\|_{L_x^p} \lesssim |t|^{-1/2}\Big\|\|u\|_{X(s)}\|u(s)\|^4_{L_x^{12}}\Big\|_{L_s^{2,1}\left( [t,\infty) \right)}.
    \end{align*}

    For $t \in [T_j,\infty)$, we decompose $[t,\infty)$ into $[t,\infty) \cap [T_{j+1},\infty)$ and $[t,\infty)\cap I_j$. Doing so, \eqref{scattering/smallness} and Proposition \ref{spacetime bounds} then imply
    \begin{align*}
        \Big\|\|u\|_{X(s)}&\|u(s)\|^4_{L_x^{12}}\Big\|_{L_s^{2,1}\left([t,\infty)\right)}\\
        & \leq \|u\|_{X(T_{j+1})}\|u\|^4_{L_s^{8,4}L_x^{12}([T_{j+1},\infty))}+ \|u\|_{X(T_j)}\|u\|^4_{L_s^{8,4}L_x^{12}(I_j )} \\
        & \leq C\big(\|u_+\|_{\dot{H}^1}\big)\|u\|_{X(T_{j+1})} + \eta^4\|u\|_{X(T_j)}.
    \end{align*}
    Combining these estimates, we find that
    \begin{align*}
        \bigg\|\int_t^\infty e^{i(t-s)\Delta}\big[|u|^4u\big](s)ds\bigg\|_{X(T_j)}
        \leq C(\|u_+\|_{\dot{H}^1})\|u\|_{X(T_{j+1})} + \eta^4\|u\|_{X(T_j)}.
    \end{align*}
    Along with \eqref{scattering/linear}, this yields the bootstrap statement \eqref{scattering/bootstrap} and concludes the proof of Theorem \ref{scattering theorem} and Theorem \ref{scattering edge case theorem}.
    
\end{proof}


\bibliographystyle{abbrv}
\bibliography{references}
\end{document}